\documentclass{amsart}

\usepackage{amsmath}
\usepackage{mathrsfs}  
\usepackage{amsthm}
\usepackage[linktocpage]{hyperref}
\usepackage{amsfonts,graphics,amsthm,amsfonts,amscd,latexsym}
\usepackage{epsfig}
\usepackage{flafter}
\usepackage{mathtools}
\usepackage{comment}
\usepackage{stmaryrd}

\hypersetup{
    colorlinks=true,    
    linkcolor=red,          
    citecolor=green,      
    filecolor=blue,      
    urlcolor=blue           
}
\usepackage{tikz}
\usetikzlibrary{graphs,positioning,arrows,shapes.misc,decorations.pathmorphing}

\tikzset{
    >=stealth,
    every picture/.style={thick},
    graphs/every graph/.style={empty nodes},
}

\tikzstyle{vertex}=[
    draw,
    circle,
    fill=black,
    inner sep=1pt,
    minimum width=5pt,
]
\usepackage[position=top]{subfig}
\usepackage[alphabetic]{amsrefs}
\usepackage{amssymb}
\usepackage{color}

\calclayout

\usetikzlibrary{decorations.pathmorphing}
\tikzstyle{printersafe}=[decoration={snake,amplitude=0pt}]

\newcommand{\Pic}{\operatorname{Pic}}

\newcommand{\dd}{\mathcal{D}}
\newcommand{\oo}{\mathcal{O}}

\newcommand{\pp}{\mathbb{P}}

\newcommand{\rr}{\mathbb{R}}

\newcommand{\kk}{\mathbb{K}}

\def\O#1.{\mathcal {O}_{#1}}			
\def\pr #1.{\mathbb P^{#1}}				
\def\af #1.{\mathbb A^{#1}}			
\def\ses#1.#2.#3.{0\to #1\to #2\to #3 \to 0}	
\def\xrar#1.{\xrightarrow{#1}}			
\def\K#1.{K_{#1}}						
\def\bA#1.{\mathbf{A}_{#1}}			
\def\bM#1.{\mathbf{M}_{#1}}				
\def\bL#1.{\mathbf{L}_{#1}}				
\def\bB#1.{\mathbf{B}_{#1}}				
\def\bK#1.{\mathbf{K}_{#1}}			
\def\subs#1.{_{#1}}					
\def\sups#1.{^{#1}}

\usepackage{tikz}
\usetikzlibrary{matrix,arrows,decorations.pathmorphing}

  \newtheorem{theorem}{Theorem}[section]
  \newtheorem{lemma}[theorem]{Lemma}
  \newtheorem{proposition}[theorem]{Proposition}
  \newtheorem{corollary}[theorem]{Corollary}

  \newtheorem{definition}[theorem]{Definition}
  \newtheorem{example}[theorem]{Example}

  \newtheorem{question}[theorem]{Question}

\theoremstyle{remark}

\numberwithin{equation}{section}

\usepackage[all]{xy}

\begin{document}

\title[Maximal log Fano manifolds are generalized Bott towers]{Maximal log Fano manifolds are generalized Bott towers}

\author[K.~Loginov]{Konstantin Loginov}
\address{Steklov Mathematical Institute of Russian Academy of Sciences, Moscow, Russia; Laboratory of Algebraic Geometry, HSE University, Russian Federation; Laboratory of AGHA, Moscow Institute of Physics and Technology.}
\email{kostyaloginov@gmail.com}


\author[J.~Moraga]{Joaqu\'in Moraga}
\address{Department of Mathematics, Princeton University, Fine Hall, Washington Road, Princeton, NJ 08544-1000, USA
}
\email{jmoraga@princeton.edu}
\maketitle

\begin{abstract}
We prove that maximal log Fano manifolds are generalized Bott towers.
As an application, we prove that in each dimension, there is a unique maximal snc Fano variety satisfying Friedman's d-semistability condition.
\end{abstract}

\setcounter{tocdepth}{1} 
\tableofcontents

\section{Introduction}
We work over an algebraically closed field $\kk$ of characteristic $0$. In this article, we study {\em log Fano manifolds}, i.e., 
log smooth pairs $(X,\Delta)$ so that $-(K_X+\Delta)$ is ample and $\Delta$ is reduced.
The classification of log Fano manifolds is known up to dimension $3$, see~\cite{Ma83}.
In~\cite{F14}, there are partial generalizations to higher dimensions.
Recall that Fano manifolds are bounded in each dimension~\cites{KMM92,B16}.
See~\cite{IP99} for a survey of a classsification up to dimension $3$.
On the other hand, log Fano manifolds are not bounded even in dimension two.
Indeed, in dimension two there are countably many isomorphism classes of such pairs, see Section~\ref{sec-surface-pairs}.
Furthermore, in dimension $3$ and higher log Fano manifolds have moduli, see Example~\ref{ex:moduli}.
In Example~\ref{ex:moduli}, the boundary divisor $\Delta$ is irreducible.
In contrast, if in the three-dimensional case the boundary has $3$ components, which is the maximal possible number in this case, then the moduli of such pairs is discrete.
We observe the following phenomenon which is the starting point of this research: 
the more components the boundary divisor has, the more control over the geometry of $X$ we have.
By~\cite{Lo19}*{3.3}, we know that $\Delta$ can have at most $\dim X$ components.
We say a log Fano manifold is {\em maximal} if this upper bound is attained.
In the literature, the terms \emph{maximal boundary} \cite{HK18} and \emph{maximal intersection} \cite{KX16} are also used. For an explanation of the latter term see Section \ref{dual_complex}. In this work, we show that the geometry of a maximal log Fano manifold can be described explicitly. It turns out that such pair is \emph{toric}, by which we mean that $X$ admits the structure of a toric variety such that the boundary divisor $\Delta$ is torus-invariant. One can say even more, namely that $X$ is a so called \emph{generalized Bott towers} as defined in \cite{CMS10}. Roughly speaking, this means that $X$ is an iterated projective bundle over a point. More precisely, there exists a sequence of contractions 
\[ 
X=X_m \to X_{m-1} \to \ldots \to X_1 \to X_0 = \mathrm{pt}
\] 
where each contraction $X_{i} \to X_{i-1}$ is a $\mathbb{P}^{n_i}$-bundle for some $n_i\geq 1$, and $X_{i} = \mathbb{P}_{X_i} (\mathscr{E}_{i})$ where $\mathscr{E}_{i}$ is a split vector bundle of rank $n_i+1$. If each $n_i$ is equal to $1$ then $X$ is called a {\em Bott tower}. See Appendix~\ref{section-bott-towers} for the details. 
Now, we are ready to formulate our main result.

\begin{theorem}
\label{theorem-1} 
Let $(X, \Delta)$ be a log smooth log Fano manifold of dimension $n$. Then $\Delta$ has at most $n$ components. If the number of components is equal to $n$, then $X$ is a generalized Bott tower and the pair $(X, \Delta)$ is toric.
\end{theorem}

As a consequence, up to isomorphism, there are countably many maximal log Fano manifolds (of any dimension). This is because there are only countably many isomorphism classes of generalized Bott towers which follows from their toric description, see Appendix~\ref{section-bott-towers}. Example~\ref{ex:not-max} and the discussion below show that the main theorem does not hold if we weaken any of its assumptions.

We briefly discuss the situation for {\em log Calabi-Yau pairs} (or {\em log CY} for short), i.e., log canonical pairs $(X,\Delta)$ for which $K_X+\Delta \equiv 0$.
In the case of log CY pairs, there is no bound for the number of components of $\Delta$ in terms of the dimension of $X$.
Indeed, pick any log CY surface $(X,\Delta)$ such that $\Delta$ is connected and has at least two components.
Then, blowing-up a point of intersection of two boundary components, we get a new log CY pair that has one more boundary component.
However, one can define the notion of a maximal intersection for such pairs by saying that $\Delta$ has a $0$-dimensional stratum.
This means that some of the components of $\Delta$ intersect in a point. In contrast to the log Fano case, maximal intersection log CY pairs need not to be toric.
Indeed, consider any log CY surface and blow-up sufficiently many general points in the boundary.
Then, the resulting surface will not be toric.
On the positive side, the log CY surface pairs with maximal intersection admit so called {\em birational toric models}, see~\cite{HK18}. However, these birational toric models do not exist starting in dimension $3$~\cite{Ka20}.
A similar argument shows that the main theorem fails for Fano type pairs, even if they are log smooth.

As an application of Theorem \ref{theorem-1}, we establish the uniqueness of maximal semistable degenerations of Fano manifolds. Recall that a \emph{semistable family} is a family of projective algebraic varieties over a curve germ such that its total space is smooth and  the special fiber is reduced and has simple normal crossings. In other words, the special fiber is a \emph{simple normal crossings variety}, or an \emph{snc variety} for short. We say that the special fiber of a semistable family is a \emph{semistable degeneration} of its generic fiber. 

The special fiber of a semistable family satisfies the $d$-semistability condition introduced by Friedman in \cite{Fr83}, see Definition~\ref{dss-defin}. On snc varieties, the dualizing sheaf is a line bundle. For such a variety $X=\sum X_i$, we say that $X$ is Fano if $-K_X$ is ample. This is equivalent to the following condition: each component $X_i$ is log Fano with respect to the boundary $\Delta_{X_i}$ induced by the intersection with the other components. If every fiber in a semistable family is a Fano variety, then we say that we have a \emph{semistable family of Fano varieties}. We call its special fiber is a 
\emph{semistable degeneration of Fano manifolds}. 

The dual complex of the special fiber is an important invariant of a degeneration. Its topology in some sense reflects the geometry of the generic fiber. There are many results along these lines. In \cite{dFKX16} it is proved that if the generic fiber of a semistable family is rationally connected, then the dual complex of the special fiber is contractible. More specifically, by \cite{Lo19}*{3.3} the dual complex of the special fiber of semistable family of Fano varieties is a simplex whose dimension $k$ does not exceed the dimension of the general fiber $n$.

We say that a semistable degeneration of Fano manifolds is \emph{maximal} if the dual complex of its special fiber has dimension $n$. In \cite{Hu06} it is shown for $n\geq 1$ any $k\leq n$ can be realized for some degeneration with general fiber isomorphic to $\mathbb{P}^n$. For $k=n$ this construction coincides with the maximal degeneration described below. In \cite{Tz15} it is proved that any $d$-semistable snc Fano variety can be smoothed, that is, embedded as the special fiber in a semistable family. Hence, the classification of semistable degenerations of Fano manifolds is equivalent to the classification of $d$-semistable snc Fano varieties. 

We describe a semistable degeneration of Fano manifolds which is maximal, see \cite{Hu06}*{4.3}.
Let $\pp^0 \subset \pp^1 \subset \dots \subset \pp^{n-2}$ be a sequence of torus invariant linear subspaces of $\pp^n$. 
Consider the succesive blow-up of $\pp^0$, the strict transform of $\pp^1$, the strict transform of $\pp^2$, and so on.
We can glue $n+1$ copies of the obtained manifolds along torus invariant strata to obtain a maximal snc Fano variety satisfying Friedman's d-semistability condition.
We call this snc Fano variety $X^n$.
The next result describes maximal semistable degenerations of Fano manifolds.

\begin{theorem}
\label{thm-2}
A maximal snc Fano variety satisfying Friedman's d-semistability condition is isomorphic to $X^n$. A maximal semistable degeneration of Fano manifolds is isomorphic to $X^n$. 
\end{theorem}

For the results on the dual complex of semistable degenerations of Calabi-Yau manifolds see \cite{KX16} and references therein.

\subsection*{Acknowledgements}
The first author was partially supported by the HSE University Basic Research Program, Russian Academic Excellence Project ``5-100'', Foundation for the Advancement of Theoretical Physics and Mathematics ``BASIS'', and the Simons Foundation. He thanks Yuri Prokhorov for reading the draft of the paper, useful discussions and helpful comments.

\section{Sketch of the proof}

We briefly explain the strategy of the proof of Theorem \ref{theorem-1}. It relies on several results. First of all, to study the geometry of the pair $(X, \Delta)$, we use the notion of a \emph{dual complex} $\mathcal{D}(\Delta)$, see Section \ref{dual_complex}. This is a topological invariant that contains the combinatorial information about the intersections of the components of $\Delta$. In \cite{Lo19} it was proved that for an $n$-dimensional log Fano manifold the dual complex is a simplex of dimension $\leq n - 1$. In particular, for a maximal log Fano manifold, $\mathcal{D}(\Delta)$ is a simplex of dimension $n-1$, so $\Delta$ has precisely $n$ components and the intersection of all of them is a point. Using this combinatorial data, we prove Theorem \ref{theorem-1} by induction starting from dimension $2$. Adjunction formula tells us that any component, and more generally any stratum $D$ of $\Delta$ is also a maximal log Fano manifold with respect to some boundary $\Delta_D$. Hence, we may apply the theorem inductively.

The crucial idea is to find an extremal smooth rational curve on $X$ such that the associated contraction is a fibration. We call such curve \emph{free} since its normal bundle is semiample. In fact, we prove that we can find such curve among $1$-dimensional strata of $\Delta$, see Proposition \ref{prop-free-curve}. At this step, we use adjunction theory (see Section \ref{section_adjunction_theory}) and a careful analysis of the contractions of $X$. The existence of such curve gives us a bound on the Picard number of $X$ in terms of the Picard number of some component of $\Delta$. Using this bound, we prove that some stratum $D$ of $\Delta$ is a section of the constructed contraction.

Then, using the classical characterization of projective spaces \ref{thm-index-charact} and \ref{thm-charact-scroll}, we conclude that each fiber of this contraction is in fact a projective space. Hence, we have a projective bundle over a log Fano manifold of lower dimension. The next step is to show that our projective bundle is a projectivization of a split vector bundle. To this aim, we use the lifting of a toric $1$-complement from $\Delta$ to $X$, see Lemma \ref{lemma-glue-gamma} and Propositions \ref{prop-log-smooth-complement} and \ref{prop-log-smooth-complement2}. Using the lifted $1$-complement, we prove that our projective bundle has sufficiently many disjoint sections which gives us the desired splitting. Thus, we conclude that $X$ has a projective bundle structure over a maximal log Fano manifold $(D, \Delta_D)$ of lower dimension. Hence, applying induction, we see that $X$ is a generalized Bott tower. Toricity of $(X, \Delta)$ follows from toricity of $(D, \Delta_D)$ and the splitting of the corresponding vector bundle. We summarize our knowledge of geometry of maximal log Fano manifolds in Theorem \ref{the-theorem}.

Our approach to prove that $(X, \Delta)$ is toric reminds of Shokurov's complexity theory. Recall that the \emph{complexity} of an $n$-dimension pair $(X, \Delta')$ is defined as 
\[
c=c(X, \Delta') := n + \rho - d,
\]
where $\rho=\rho(X)$ is the Picard number of $X$ and $d$ is the sum of coefficients of $\Delta'$. In \cite{BMSZ18} it is proved that if $-K_X-\Delta'$ is nef then $c\geq 0$. Furthermore, if $c=0$, then the pair $(X, \Delta')$ is toric. For a maximal log Fano manifold, if we can find an $1$-complement $\Gamma$ of $K_X+\Delta$ which has $\rho$ components, then we can take $\Delta'=\Delta + \Gamma$ and conclude that $c=0$. Hence, $(X, \Delta)$ is toric. However, it turns out that, a priori, there is no reason for an arbitrary $1$-complement $\Gamma$ to have as many components as we wish, see Example \ref{rem-toric-compl}. Also, a priori, there is no bound for $\rho$.
This leads us to apply combinatorial techniques and the adjunction theory to tackle the problem. 

\section{Preliminaries}
All the varieties in this paper are projective over the base field unless stated otherwise. We use standard notation and terminology of the minimal model program, see, e.g., \cite{KMM87}. 

\subsection{Pairs and contractions} By a {\em contraction}, we mean a projective morphism $f \colon X \rightarrow Y$ of varieties such that $f_*\oo_X = \oo_Y$. In particular, $f$ is surjective and has connected fibres. Moreover, if X is normal, then Y is also normal. Recall that the exceptional locus of $f$ is the union of all subvarities $Z\subset X$ such that $\dim f(Z)<\dim Z$. We denote the exceptional locus of $f$ by $\mathrm{Exc}\,(f)$. The contraction $f$ is said to be a \emph{fibration} or \emph{of fiber type} if $X \subset \mathrm{Exc}\,(f)$. We say that a subvariety $Z\subset X$ is {\em contracted} if $Z\subset \mathrm{Exc}\,(f)$. 

A \emph{pair} $(X, \Delta)$ consists of a normal variety $X$ and a boundary $\mathbb{Q}$-divisor $\Delta$ with coefficients in $[0, 1]$ such that $K_X + \Delta$ is $\mathbb{Q}$-Cartier. Let $f \colon Y \to X$ be a birational contraction. If $K_X + \Delta$ is $\mathbb{Q}$-Cartier, we define the $\mathbb{Q}$-divisor $\Delta_Y$ called the \emph{log pull-back} of $\Delta$ by the formulas: 
\[
K_Y +\Delta_Y \sim_{\mathbb{Q}} f^*( K_X + \Delta ) \text{ and } f_* \Delta_Y = \Delta.
\]
For a divisor $E \subset Y$ we define its \emph{discrepancy} $a(E, X, \Delta)$ as minus its coefficient in $\Delta_Y$. A pair $(X, \Delta)$ is \emph{log canonical}, or \emph{lc} for short, if $a(E, X, \Delta) \geq -1$ for every divisor $E$ and every birational contraction $f \colon Y \to X$.
Let $(X, \Delta)$ be an lc pair. An irreducible subvariety $Z \subset X$ is an \emph{lc center} if there is a birational contraction $f \colon Y \to X$ and a  divisor $E\subset Y$ such that $a(E, X, \Delta)=-1$ and $f(E) = Z$.

\subsection{Log Fano manifolds}
\label{subsection-log-fano-pairs}{
We say that a pair $(X, \Delta)$ is log smooth if $X$ is smooth and the support of $\Delta$ has simple normal crossings. In particular, every component of $\Delta$ is smooth.

\begin{definition}{\em 
Throughout the paper, by a \emph{log Fano manifold}, we mean a projective log smooth pair $(X, \Delta)$ where $\Delta$ is an integral effective divisor and $-K_X - \Delta$ is ample. We use the words log Fano manifolds and \emph{log Fano pairs} interchangeably. Also, if for a smooth variety $X$ there exists a boundary divisor $\Delta$ such that $(X, \Delta)$ is a log Fano manifold then we call $X$ a \emph{log Fano manifold}. If $(X, 0)$ is a log Fano manifold, $X$ is called a \emph{Fano manifold}.
}
\end{definition}

Let $(X, \Delta)$ be a log Fano manifold. By the main result in \cite{Zh06} applied to the pair $(X, (1-\epsilon)\Delta)$ for some $0<\epsilon\ll 1$ we see that $X$ is rationally connected. Let $\Delta = \sum_{i=1}^{k} D_i$ where $D_i$ are prime divisors. By a \emph{stratum $D$} of $\Delta$, we mean an intersection of the form $D=D_{i_1}\cap \ldots \cap D_{i_{m}}$ for some $m$ where $0\leq i_1<\ldots < i_m \leq k$ and $0\leq m\leq k$. If $m=0$, we let $D$ to be equal to $X$ which we also call a stratum of $\Delta$. For a log Fano manifold $(X, \Delta)$ any stratum of  $\Delta$ is smooth and irreducible, see Theorem \ref{thm-dual-complex}. Moreover, by the adjunction formula, the pair
\[
\left( D,\ \ \Delta_D = \sum_{j \neq i_1, \ldots, i_{m}} D_j|_{D} \right)
\]
is a log Fano manifold. In particular, $D$ is rationally connected, hence if $\dim D\leq 2$, then $D$ is rational. By \cite{Ma83}*{\S 2}, we have that
\[
\Pic(X)\simeq \mathrm{H}^2(X, \mathbb{Z}),\quad \quad h^{i}(\oo_X) = 0\ \ \text{for}\ \ i>0,
\] 
and $\Pic(X)$ is torsion free. Moreover, $\Pic(X)\otimes \mathbb{R} = \mathrm{N}^1(X)$ since linear equivalence of divisors coincides with numerical equivalence. The Mori cone $\overline{\mathrm{NE}}(X)$ for a log Fano manifold $(X, \Delta)$ is polyhedral and $(K_X+\Delta)$-negative. In particular, for any extremal face of $\overline{\mathrm{NE}}(X)$ there exists the corresponding contraction which contracts precisely the curves whose classes belong to that extremal face~\cite{KMM87}*{3.2.1}. }

\begin{example}\label{ex:moduli}{\em 
In this example, we show that log Fano manifolds in dimension three can have moduli.
Consider the pair $(X,\Delta)$ where $X$ is a smooth $3$-dimensional cubic hypersurface in $\pp^4$ and $\Delta$ is a general hyperplane secction.
In this example, the boundary divisor $\Delta$ is irreducible and the class of such pairs $(X,\Delta)$ has positive dimensional moduli.
}
\end{example}

\subsection{Dual complex}{
\label{dual_complex}
\emph{The dual complex}, denoted by $\mathcal{D}(\Delta)$, of a simple normal crossing divisor $\Delta=\sum D_i$ on a smooth variety $X$ is defined as follows: $\mathcal{D}(\Delta)$ is a regular cell complex whose vertices are in one-to-one correspondence with the irreducible components $D_i$ of $\Delta$. Then, its $m$-cells correspond bijectively to the irreducible components of the intersection of $m + 1$ irreducible components $D_{i_1}\cap \ldots \cap D_{i_{m+1}}$ for $i_1 < \ldots < i_{m+1}$ (they are also called strata of~$\mathcal{D}(\Delta)$). The attaching maps are defined in the natural way. Obviously, the dimension of $\mathcal{D}(\Delta)$ does not exceed the dimension of $\Delta$. See Section $2$ of \cite{dFKX16} for details.

\begin{theorem}[\cite{Lo19}, Theorem 3.3]
\label{thm-dual-complex}
Let $(X, \Delta)$ be a log Fano manifold. Then the dual complex $\dd(\Delta)$ of $\Delta$ is a simplex of dimension $\leq \dim X - 1$. In particular, $\Delta$ has at most $\dim X$ components.
\end{theorem}

In fact, the above theorem holds if the pair $(X, \Delta)$ has at worst dlt singularities. The next definition is the main in this paper.  

\begin{definition}{\em 
A log Fano manifold $(X, \Delta)$ is called \emph{maximal}, if the following equivalent conditions are satisfied:
\begin{enumerate}
\item
the boundary $\Delta$ has $\dim X$ components,
\item
$\dim \dd(\Delta)=\dim X - 1$,
\item
$\Delta$ has a $0$-dimensional stratum, or
\item
the pair $(X, \Delta)$ has a $0$-dimensional lc center.
\end{enumerate} 
}
\end{definition}

\subsection{Wi\'sniewski's inequality}
Let $X$ be a smooth projective variety and let $f$ be a $K_X$-negative contraction of an extremal ray $R$ on $X$. Recall that the \emph{length} of the extremal ray $R$, denoted by $l(R)$, is defined as follows:
\[
l(R) := \mathrm{min} \{(- K_X) \cdot C\ |\ [C]\in R \},
\]
where $C$ is a rational curve and $[C] \in R$ means that its numerical equivalence class belongs to $R$. By~\cite{W91}, the following inequality holds:
\[
\dim \mathrm{Exc}\,(f) + \dim F \geq \dim Y + l(R) - 1
\]
where $F$ is a fiber of $f$ such that $F\subset\mathrm{Exc}\,(f)$.

\section{Adjunction theory}
\label{section_adjunction_theory}
By a \emph{polarized variety}, we mean a smooth projective variety $X$ together with an ample line bundle $L$ on it. Abusing notation, sometimes we will denote divisors and the corresponding line bundles by the same letter.

\begin{definition}{\em 
Let $X$ be a variety polarized by an ample line bundle $L$. Assume that $K_X$ is not nef. For $t \gg 0$, $K_X + tL$ is an ample $\mathbb{R}$-divisor. The \emph{nef value} of $L$ denoted by $\tau=\tau(L)$ is a real number
\[
\tau(L)=\inf\{t\geq0\;|\;K_X +tL\in \overline{\mathrm{Nef}}(X)\}
\]
where $\overline{\mathrm{Nef}}(X)$ is the closure of the cone of nef divisors on $X$.}
\end{definition}

Due to Kawamata-Shokurov basepoint free theorem \cite{KMM87}*{3.1.1}, for a nef value $\tau$ there exists a contraction morphism $\psi$, called the \emph{nef value morphism}, given by the linear system $|m(K_X + \tau L)|$ for some $m\gg 0$ and $\tau \in \mathbb{Q}$.

Recall the following results that characterize the projective space and projective bundles. We will formulate them for smooth varieties even though they hold in a more general setting.

\begin{theorem}[\cite{KO70}]
\label{thm-index-charact}
Let $X$ be a smooth Fano manifold of dimension $n$. The \emph{index} of $X$ is the maximal integer $r$ such that $-K_X=rH$ for some $H\in \Pic (X)$. Then $r\leq n+1$ and if $r=n+1$, then $X$ is isomorphic to the projective space $\mathbb{P}^n$. 
\end{theorem}
\begin{corollary}
\label{lem-rho-1}
If for a maximal log Fano manifold $(X, \Delta)$ one has $\rho(X)=1$ then $X=\mathbb{P}^n$ and each $D_i$ is a hyperplane.
\end{corollary}
A smooth variety $X$ polarized by an ample divisor $L$ is a \emph{projective bundle} over a smooth variety $Y$ if 
\[
(X, L)\simeq (\mathbb{P}_Y(\mathscr{E}), \oo_{\mathbb{P}_Y(\mathscr{E})}(1))
\] 
holds for some vector bundle $\mathscr{E}$ on $Y$.

\begin{theorem}[\cite{F87}, 2.12]
\label{thm-charact-scroll}
Let $\pi \colon X \to Y$ be a surjective morphism from a smooth variety $X$ onto a normal projective variety $Y$. Let $L$ be an ample line bundle on $X$ and suppose that $(F, L_F)\simeq (\mathbb{P}^k, \oo_{\mathbb{P}^k}(1))$ for a general fiber $F$ of $f$. Suppose further that $\dim F = r$ for every fiber $F$ of $\pi$. Then $Y$ is smooth and $\pi$ makes $X$ a projective bundle over $Y$.
\end{theorem}

\subsection{Polarized log Fano manifolds.}
\label{subsection-polarized-log-fano-pairs}
We discuss a natural polarization on log Fano manifolds. Let $(X, \Delta)$ be a log Fano manifold of dimension $n$. Then $X$ is polarized by the divisor $L=-K_X-\Delta$ which is ample by definition. Moreover, each stratum 
\[
D=D_{i_1}\cap\ldots\cap D_{i_k}, \ \ \ \ \ 1\leq i_1< \ldots < i_k\leq n,\ \ \ \ \ 1\leq k\leq n
\] 
of $\Delta$ can be considered as a polarized variety together with the following ample divisor defined by adjunction: 
\[
L_D=-K_{D}-\Delta_{D} = (-K_X -\Delta) |_{D}.
\] 
Now, we discuss the properties of nef value morphism for maximal log Fano manifolds $(X, \Delta)$ polarized by an ample line bundle $L=-K_X-\Delta$. First of all, we show that $K_X$ is not nef. Indeed, if $\Delta=0$ then this is obvious. Assume that $\Delta\neq0$. If $K_X$ is nef then $-\Delta=K_X + L$ is nef which is absurd since $\Delta\neq0$. Let $\tau>0$ be a nef value with respect to an ample divisor $L$ and let $\psi$ be the nef value morphism given by the linear system $|m(K_X+\tau L)|$ for $m\gg0$. In general, $\psi$ may contract a positive-dimensional extremal face of $\overline{\mathrm{NE}}\, (X)$. Let $R$ be an extremal ray lying in that extremal face and let $\phi=\phi_R$ be the corresponding contraction. Let $C$ be an effective rational curve with numerical equivalence class in the ray $R$ such that the length $l(R)$ is attained on $C$, that is $l(R) = (-K_X)\cdot C$.

\begin{proposition} 
\label{inequalities}
In the setting \ref{subsection-polarized-log-fano-pairs}, if $\Delta\neq 0$ then the following assertions hold:
\begin{enumerate}
\item \label{ineq-1}
$K_{X}\cdot C<0$, 
\item \label{ineq-2}
$\tau>1$,
\item \label{ineq-3}
$\Delta \cdot C>0$ and $l(R) = (-K_{X})\cdot C \geq 2$,
\item \label{ineq-4}
If $\phi$ is birational then the dimension of any fiber of $\phi$ is $\geq 2$.
\end{enumerate}
\end{proposition}
\begin{proof}
\begin{enumerate}
\item By definition, we have $(K_{X}+\tau L)\cdot C = 0$, $\tau>0$ and $L\cdot C > 0$ since $L$ is ample.
\item Above we have shown that $K_X$ is not nef. The divisor $K_X + L = - \Delta\neq 0$ is not nef as well. Then $K_X+\tau' L$ also cannot be nef for $0 < \tau' < 1$ since the nef cone is convex and $K_X + \tau' L$ is nef for $\tau' \gg 0$.
\item We have $(K_{X}+\tau L)\cdot C = ((1-\tau)K_{X} - \tau \Delta)\cdot C = 0$, and since $\tau>1$ and $K_X\cdot C < 0$, we have $\Delta \cdot C > 0$. Consequently, $l(R) = -K_X\cdot C\geq 2$.
\item Follows from Wi\'sniewsky's inequality and the fact that $l(R)\geq 2$.
\end{enumerate}
\end{proof}

\subsection{Two-dimensional log Fano manifolds.}
\label{sec-surface-pairs}
Consider a two-dimensional log Fano manifold $(X, \Delta)$. Assume that $\Delta\neq 0$. Let $\tau$ be the nef value with respect to $L=-K_X-\Delta$, and $\psi$ be the corresponding nef value morphism. Let $\phi=\phi_R$ be a $(K+\tau L)$-trivial contraction of an extremal ray. Then by Proposition \ref{inequalities}, $\phi$ is of fiber type. By Theorem \ref{thm-dual-complex}, the boundary $\Delta$ is either irreducible, or has two components. We are interested in the case of maximal log Fano manifolds, so we assume that $\Delta=D_1+D_2$. Consider two cases. 

If $\dim \phi(X) = 0$ then $\phi$ is a $K_X$-negative extremal contraction to a point. In this case, $X=\mathbb{P}^2$, and it is clear that $\Delta$ is a union of two distinct lines. In this case we have $\tau=3$. 

If $\dim \phi(X) = 1$ then $\phi$ is a $K_X$-negative extremal contraction to a smooth rational curve. Thus, the fibers of $\phi$ are also smooth rational curves, and hence $X$ is a Hirzebruch surface $\mathbb{F}_n$ for some $n\geq 0$. By Proposition \ref{inequalities}(\ref{ineq-3}) there exists a $\phi$-horizontal component $D_1$ of $\Delta$, and since $L$ remains ample restricted to a fiber of $\phi$, there exists exactly one such component. Therefore, the second component $D_2$ of $\Delta$ coincides with a fiber of $\phi$. If $n=0$ we get a pair $(\mathbb{P}^1\times \mathbb{P}^1, l_1+l_2)$ where $l_i$ are distinct rulings. Suppose that $n>0$, and let $s$ be a $(-n)$-section of $\phi$. Assume that $s\neq D_1$, then $\Delta \cdot s > 0$. Since $L$ is ample, we have 
\[
1 \leq L\cdot s = ( -K_X - \Delta )\cdot s < - K_X \cdot s = 2 + s^2 \leq 1
\]
which is a contradiction. Hence we get a pair $(\mathbb{F}_n, s + f)$ where $s$ is a $(-n)$-section and $f$ is a fiber. In this case we have $\tau=2$. It is not difficult to classify two-dimensional log Fano pairs $(X, \Delta)$ with irreducible $\Delta$ using similar arguments. See \cite{Ma83}*{\S 3} for a different approach.

\

In higher dimensions, the nef value morphism can be either birational, or a fibration. Also it can be a contraction of an extremal ray or of an extremal face of positive dimension in the Mori cone $\overline{\mathrm{NE}}(X)$. We present some examples.
\begin{example}{\em 
Let $X$ be the blow-up of a point on $\mathbb{P}^n$ for $n\geq 2$. Let $E$ be the exceptional divisor, and $H$ be the class of general hyperplane section on $\mathbb{P}^n$ pulled back to $X$. Put
\begin{equation}
\begin{split}
\begin{gathered}
\quad K_X = -(n+1)H + (n-1) E, \quad \Delta = D_1 + \ldots + D_n, \\ 
D_1 = E, \quad D_2 \sim \ldots \sim D_n \sim H - E, \quad L = -K_X - \Delta \sim 2H - E,
\end{gathered}
\end{split}
\end{equation} 
so $D_i$ for $2\leq i\leq n$ are strict transforms of hyperplanes passing through the blown-up point on $\mathbb{P}^n$. We have $\rho(X)=2$, so the Mori cone of $X$ has two extremal rays; one of them is a divisorial contraction back to $\mathbb{P}^n$, the other gives a contraction of fiber type, more precisely, it is $\mathbb{P}^1$-bundle over $\mathbb{P}^{n-1}$. Let $s$ be a line in $E\simeq \mathbb{P}^{n-1}$, and let $f$ be a fiber of the $\mathbb{P}^1$-bundle. Then $\overline{\mathrm{NE}}(X)=\langle s, f\rangle$. One has $L\cdot f = L\cdot s = 1$, so $(X, \Delta)$ is a log Fano pair. One checks that $K_X\cdot f = -2$, and $K_X \cdot s = - ( n - 1 )$. Since the nef cone is dual to the Mori cone, to determine the nef value $\tau$ it is enough to consider two conditions: $(K_X + \tau L)\cdot f = -2 + \tau \geq 0$ and $(K_X + \tau L)\cdot s = - ( n - 1 ) + \tau \geq 0$. Thus, we have $\tau = \mathrm{max}\,(2, n - 1)$. Consider three cases:
\begin{enumerate}
\item
If $n=2$, then $\psi$ is extremal and of fiber type. Indeed, one has $\tau = 2$ and thus $K_X + \tau L = H-E$. The linear system $|H-E|$ gives a nef value morphism $\psi$ of fiber type which is a $\mathbb{P}^1$-bundle over $\mathbb{P}^1$. 
\item
If $n=3$, then $\psi$ is of fiber type but not extremal. Indeed, one has $\tau = 2$ and $K_X+\tau L = 0$. Both extremal rays are $(K_X+\tau L)$-trivial and $\psi$ is a contraction to a point.
\item
If $n\geq4$, then $\psi$ is extremal and birational. Indeed, one has $\tau=n-1$, so the linear system $|K_X+\tau L|=|(n-3)H|$ gives a birational contraction to $\mathbb{P}^n$.
\end{enumerate}}
\end{example}

\section{Geometry of maximal log Fano manifolds}
\label{section-geometry}
In this section we begin to study the geometry of maximal log Fano manifolds $(X, \Delta)$. The main facts are summarised in Theorem \ref{the-theorem}. It contains some properties of the Mori cone and the Picard group of $X$, the fact that such pairs are toric, existence and uniqueness of toric $1$-complement, and the description of extremal contractions. To formulate this theorem, we need a notion of \emph{stratum $L$-lines}. Later, we will show that there exists a stratum $L$-line whose deformations cover $X$. 

\subsection{Stratum $L$-lines}{
\label{subsection-stratum-lines}
By an \emph{$L$-line} on a variety $X$ polarized by an ample divisor $L$ we mean a smooth rational curve $l$ such that $L\cdot l = 1$. By a \emph{stratum $L$-line} on a maximal log Fano manifold $(X, \Delta)$ of dimension $n$, we mean an $L$-line on $X$, where $L=-K_X-\Delta$ with the following property: $l$ coincides with some $1$-dimensional stratum of $\Delta$. By Theorem \ref{thm-dual-complex}, the dual complex $\mathcal{D}(\Delta)$ is a simplex, so we have
\[
l = D_1\cap \ldots \cap \widehat{D_i} \cap \ldots \cap D_n
\]
(the hat means that we omit the corresponding component) for some $1\leq i\leq n$. In fact, if $l$ coincides with such stratum, it is automatically an $L$-line. Indeed, in this case by adjunction $l$ is a log Fano manifold itself, hence $l \simeq \mathbb{P}^1$ and $L\cdot l = 1$. Combinatorially, stratum $L$-lines correspond to faces of maximal dimension of the simplex $\mathcal{D}(\Delta)$, so there are $n$ of them. On any stratum $(D, \Delta_D)$ of dimension $\geq 2$, considered as a maximal log Fano manifold polarized by the divisor $L_D=-K_D-\Delta_D$, a stratum $L_{D}$-line can be defined analogously. 

We call an $L$-line $l$ \emph{free} on $X$ if its normal bundle $\mathscr{N}_{l/X}$ is \emph{semiample}, that is, $\mathscr{N}_{l/X}$ splits as the sum of the line bundles $\oo_{\mathbb{P}^1}(a_i)$ on $l\simeq \mathbb{P}^1$ for $a_i\geq 0$. By deformation theory (see, e.g., Theorem $1.1$ in \cite{Ha10}), if $l$ is free on $X$ then the deformations of $l$ cover $X$. Note that any curve which is a deformation of an $L$-line is irreducible and reduced since by definition $L\cdot l = 1$.

The next theorem is one of main results of this paper. We say that the pair $(X, \Delta)$ is \emph{toric} if $X$ admits the structure of a toric variety such that the boundary $\Delta$ is torus-invariant. We say that an effective curve is \emph{extremal} if its numerical equivalence class belongs to an extremal ray.

\begin{theorem} 
\label{the-theorem}
Let $(X, \Delta)$ be a maximal log Fano manifold of dimension $n\geq2$. Put $L = -K_X - \Delta$. Then the following assertions hold:
\begin{enumerate}
\item \label{thm-toric-pair}
the pair $(X, \Delta)$ is toric,
\item \label{thm-gamma}
the divisor $K_X + \Delta$ admits a unique toric log smooth $1$-complement $\Gamma$, that is an integral effective divisor $\Gamma$ such that $K_X + \Delta + \Gamma \sim 0$. Moreover, such $\Gamma$ has $\rho(X)$ components,
\item \label{thm-stratum-lines}
the Mori cone $\overline{\mathrm{NE}}(X)$ is generated by stratum $L$-lines, every stratum $L$-line is extremal, and any extremal curve is proportional to some stratum $L$-line,
\item \label{thm-picard}
\label{thm-picard-bound}
the Picard group of $X$ is generated by the components $D_i$ of $\Delta$. In particular, $\rho(X)\leq n$,
\item \label{thm-nef-component}
there exists a component $D_i$ of $\Delta$ such that $|D_i|$ is basepoint free. In particular, $D_i$ is nef on $X$, and $\mathscr{N}_{D_i/X} = \oo_{D_i}(D_i|_{D_i})$ is nef on $D_i$,
\item \label{thm-stratum-set}
for any extremal contraction $\pi$ its exceptional locus $\mathrm{Exc}\,(\pi)$ coincides with some stratum of $\Delta$\footnote{Recall that by our convention $X$ is also a stratum of $\Delta$.}, and there exists a fiber $F$ of $\pi$ that coincides with some stratum of $\Delta$,
\item \label{thm-free-curve}
there exists a free stratum $L$-line $l$, that is, $\mathscr{N}_{l/X}$ is semiample, 
\item \label{thm-projective-bundle}
contraction of any free stratum $L$-line induces a $\mathbb{P}^k$-bundle structure on $X$ for some $1\leq k\leq n$. Moreover, if, up to reordering, $D_1$ is a unique component of $\Delta$ such that $l\not\subset D_1$, then  
\begin{enumerate}
\item
either $k=1$, $\rho(X)=\rho(D_1)+1$, $D=D_1$ is a section, and $D_2,\dots,D_n$ are composed of fibers, or
\item
or $2\leq k\leq n$, $\rho(X)=\rho(D_1)$, and, up to reordering, the stratum $D=D_1\cap\ldots \cap D_k$ is a section, and $D_{k+1},\ldots, D_n$ are composed of fibers, and 
\end{enumerate}
\item \label{thm-split-bundle}
we have $X \simeq \mathbb{P}_D (\mathscr{E})$ where $D$ is the section from (\ref{thm-projective-bundle}) and $\mathscr{E}$ is a split vector bundle on $D$. In particular, $X$ is a projective bundle over a maximal log Fano manifold $(D, \Delta_D)$ of lower dimension. 
\end{enumerate}
\end{theorem}

As a consequence, we have a description of all extremal contractions on $X$. Indeed, by Theorem \ref{the-theorem} any extremal contraction on $X$ is a contraction of some stratum $L$-line. If the corresponding stratum $L$-line is free, we have a contraction of fiber type which is a projective bundle. If the contraction is birational, using decomposition \eqref{eq-normal-bundles} we conclude that it is of fiber type on some stratum. Hence it contracts precisely this stratum on which it induces a projective bundle structure.

\begin{corollary}
\label{thm-bott-tower}
$X$ is a generalized Bott tower (see Appendix \ref{section-bott-towers} for definition) with $\rho(X)$ stages.
\end{corollary}
\begin{proof}
By Theorem \ref{the-theorem} we have that $X$ is a projectivization of a split vector bundle over a maximal log Fano manifold $(D, \Delta_D)$ of lower dimension. Hence by definition it is enough to show that $D$ is a generalized Bott tower, see Appendix \ref{section-bott-towers}. Proceeding this way, we arrive at projective bundle over a point, that is, a projective space, which is obviously a generalized Bott tower.   
\end{proof}

\begin{corollary}
\label{cor-bott-tower}
$\rho(X)=n$ if and only if $(X, \Delta)$ is a Bott tower.
\end{corollary}
\begin{proof}
By Theorem \ref{the-theorem} the case $\rho(X)=n$ is possible if and only if at each stage of a generalized Bott tower we have a $\mathbb{P}^1$-bundle. Hence $X$ is a Bott tower according to the definition (see Appendix \ref{section-bott-towers}).
\end{proof}

\begin{example}\label{ex:not-max}{\em 
In this example, we show that the main theorem does not hold if $(X,\Delta)$ is not maximal.
More precisely, the toricity of $(X,\Delta)$ may not hold. Indeed, let $X$ be a quadric hypersurface in $\mathbb{P}^{n+1}$ for $n\geq 3$, and let $\Delta$ be a union of $n-1$ general hyperplane sections of $X$. Then $-(K_X+\Delta)$ is ample, but it is well known that $X$ is not toric.}
\end{example}

We will prove Theorem \ref{the-theorem} by induction on the dimension of $X$ starting from dimension $2$. By the classification of $2$-dimensional maximal log Fano manifolds in Section \ref{sec-surface-pairs}, we have two possibilities for $(X, \Delta)$, $\dim X = 2$: 
\begin{enumerate}
\item
$(\mathbb{P}^2, l_1 + l_2)$ where $l_1$ and $l_2$ are two different lines, or 
\item
$(\mathbb{F}_n, s + f)$ where $n\geq 0$, $s$ is the $(-n)$-section and $f$ is a fiber of the projection $\mathbb{F}_n\to \mathbb{P}^1$. 
\end{enumerate}
One easily checks that Theorem \ref{the-theorem} holds for these pairs, so we will assume that $n\geq 3$ and the theorem holds for any maximal log Fano manifold of dimension $\leq n - 1$. In the rest of this section we prove some preliminary lemmas. Sections \ref{sect-looking-for-a-line}, \ref{sect-first-case} and \ref{sect-second-case} contain the proofs of the main propositions needed for Theorem \ref{the-theorem}. Then, in Section \ref{sect-proof-of-thm} we finish the proof. 

\begin{lemma}
\label{lem-fiber-type}
Let $(X, \Delta)$ be a maximal log Fano manifold of dimension $n$. Assume that $\phi$ is an extremal contraction of fiber type on $X$. Then $\phi$ induces a contraction of fiber type on some component $D_i$ of $\Delta$. 
\end{lemma}
\begin{proof}
Let $R$ be an extremal ray that corresponds to $\phi$, and let $C$ be an effective curve in $R$. If for some component $D_i$ we have $D_i\cdot C = 0$, then $\phi$ induces a contraction of fiber type on $D_i$ (possibly, after passing to Stein factorization), and we are done. So assume $D_i\cdot C>0$ for any $i$, $1\leq i\leq n$. Consider a smooth general fiber $F$ of $\phi$.
By adjunction, we have that
$K_X|_F=K_F$, then we have
\[
K_X + \Delta|_F = K_F + \Delta_F 
\]
and the divisor $\Delta_F = \sum D_i|_F$ has at least $n$ components. Since the divisor $-K_F - \Delta_F$ is ample, we conclude that $F$ is a log Fano manifold. Moreover, since the contraction $\phi$ is extremal, the components of $\Delta_F$ are all proportional to $K_F$ in $\Pic (F)$. Hence $F$ is a Fano manifold of dimension $\leq n$ and index $\geq n+1$. By Theorem \ref{thm-index-charact} the index of $F$ is exactly $n+1$ and $F\simeq \mathbb{P}^n$ . Thus, $F = X$, each $D_i$ is a hyperplane, and $\phi$ contracts $X$ and each component $D_i$ to a point. Hence it is of fiber type on each component $D_i$. The proof is complete. 
\end{proof}

In the next lemma, we consider the boundary divisor $\Delta$ as an \emph{snc variety}. For basic properties of snc varieties see, e.g., Section $2$ in \cite{F14}.

\begin{lemma}
\label{lemma-glue-gamma}
Let $(X, \Delta)$ be a maximal log Fano manifold of dimension $n$. Then there exists an $1$-complement $\Gamma_\Delta$ of $K_\Delta$ on $\Delta$ such that it restricts to a toric $1$-complement $\Gamma_D$ of $K_D + \Delta_D$ on each stratum $D$ of $\Delta$. 
\end{lemma}
\begin{proof}
Consider a component $D_i$ of $\Delta$ for some $i$. We know that $(D_i, \Delta_{D_i})$ is a maximal log Fano manifold of dimension $n-1$, hence by Theorem \ref{the-theorem} it is toric and there exists a unique log smooth toric $1$-complement $\Gamma_{D_i}$ of $K_{D_i}+\Delta_{D_i}$. In fact, $\Gamma_{D_i}$ is a union of all torus-invariant divisors on $D_i$ different from those lying in the support of $\Delta_{D_i}$.  Since every stratum $D$ of $\Delta$ is a maximal log Fano manifold, by Theorem \ref{the-theorem} it is toric, and so by adjunction $\Gamma_{D_i}$ restricts to a unique toric $1$-complement $\Gamma_D$ on $D$ for any $D\subset D_i$. Thus $\Gamma_{D_i}$ and $\Gamma_{D_j}$ are defined in such a way that they coincide on $D_i\cap D_j$ for any $i, j$. Therefore they define a divisor $\Gamma_\Delta$ on $\Delta$ which is an $1$-complement of $K_\Delta$ as in the statement of the lemma, and the proof is complete.
\end{proof}

\section{Looking for a line}
\label{sect-looking-for-a-line}
We work in the notation of Section \ref{section-geometry}. In particular, as before, $(X, \Delta)$ is a maximal log Fano manifold of dimension $n$. In this section we find a free stratum line on $X$, see Section \ref{subsection-stratum-lines}, and then study the space $\mathrm{N}_1(X)$. To do this, first we describe the normal bundle to a stratum line.  

\subsection{Normal bundle to a stratum $L$-line.}
\label{subsect-normal-bundle}
We use notation \ref{subsection-stratum-lines}. We show that for any stratum $L$-line $l$ on $X$, its normal bundle $\mathscr{N}_{l/X}$ splits in a natural way. Up to reordering of the components of $\Delta$, we may assume that $l = D_2\cap \ldots \cap D_{n}$. Define smooth irreducible surfaces 
\[
S_{1,i} = D_2 \cap\ldots \cap \widehat{D_i} \cap \ldots \cap D_{n}
\text{ for each } 2\leq i\leq n
\]
that coincide with strata of $\Delta$ and contain the line $l$ for any $i$. Note that 
\[
l = S_{1, 2}\cap \ldots \cap S_{1, n}.
\]
Thus the normal bundle to $l$ in $X$ splits as follows:
\begin{equation}
\begin{split}
\mathscr{N}_{l/X} &\simeq \mathscr{N}_{l/ S_{1,2}} \oplus \ldots \oplus \mathscr{N}_{l/S_{1, n}}.
\end{split} 
\end{equation}
Analogously, for any $i$, $2\leq i \leq n$, we have
\begin{equation}
\begin{split}
\mathscr{N}_{l/D_i} &\simeq \mathscr{N}_{l/S_{1,2}} \oplus \ldots \oplus \widehat{\mathscr{N}_{l/S_{1,i}}}\oplus \ldots \oplus \mathscr{N}_{l/S_{{1, n}}}.\\
\end{split} 
\end{equation}
Therefore
\begin{equation}
\begin{split}
\mathscr{N}_{l/X }&\simeq \mathscr{N}_{l/D_i} \oplus \mathscr{N}_{l/S_{1,i}}
\end{split} 
\end{equation}
Hence we see that the exact sequence
\begin{equation}
\begin{split}
0 \to \mathscr{N}_{l/D_i} \to \mathscr{N}_{l/X} \to \mathscr{N}_{D_i/X}|_l \to 0
\end{split} 
\end{equation}
splits and we have 
\begin{equation}
\begin{split}
\mathscr{N}_{l/S_{1, i}} \simeq \mathscr{N}_{D_i/X}|_l,& \quad \quad \mathscr{N}_{l/X } \simeq \mathscr{N}_{l/D_i} \oplus \mathscr{N}_{D_i/X}|_l.
\end{split} 
\end{equation}
Note that $\deg \mathscr{N}_{D_i/X}|_l = D_i \cdot l$. To sum up, we have
\begin{equation}
\label{eq-normal-bundles}
\begin{split}
\begin{gathered}
\mathscr{N}_{l/X} \simeq \mathscr{N}_{D_2/X}|_l \oplus \ldots \oplus \mathscr{N}_{D_{n}/X}|_l,\\
\mathscr{N}_{l/D_i} \simeq \mathscr{N}_{D_2/X}|_l \oplus \ldots \oplus \widehat{\mathscr{N}_{D_i/X}|_l} \oplus \ldots \oplus \mathscr{N}_{D_{n}/X}|_l.
\end{gathered}
\end{split}
\end{equation}

It is easy to write down a similar decomposition for the normal bundle to any stratum $L$-line inside any stratum of $\Delta$. We have proved the following
\begin{proposition}
\label{prop-free-criterion}
A stratum $L$-line $l$ on $X$ is free if and only if 
\[
\deg \mathscr{N}_{D_i/X}|_l \geq 0
\] 
for any $D_i$ such that $l \subset D_i$. This is equivalent to the condition that $D_i\cdot l\geq 0$ for any $D_i$ such that $l \subset D_i$. In particular, if $l$ is free on $X$ then it is free on any stratum $D$ of $\Delta$ such that $l\subset D$.
\end{proposition} 

As an immediate corollary of the above discussion, we obtain the following statement.

\begin{lemma}
\label{lemma-two-free-lines}
If a stratum $L$-line is free on two different components of $\Delta$, then it is free on $X$.
\end{lemma}
\begin{proof}
We may assume that we have a stratum $L$-line $l = D_2\cap \ldots \cap D_n$ that is free on $D_2$ and $D_3$. Then from decomposition \eqref{eq-normal-bundles} it follows that $\deg \mathscr{N}_{D_i/X}|_l \geq 0$ for any $i$, $2\leq i \leq n$, and hence $l$ is free on $X$ by Proposition \ref{prop-free-criterion}.
\end{proof}
}

The next proposition is one of the key results in this paper. 

\begin{proposition}
\label{prop-free-curve}
There exists a free stratum $L$-line $l$ on $X$. 
\end{proposition}
\begin{proof}
By induction, Theorem \ref{the-theorem}(\ref{thm-free-curve}) holds for each boundary component $(D_i, \Delta_{D_i})$ considered as a maximal log Fano manifold. Put $L_i = -K_{D_i}-\Delta_{D_i}$. So there exists a stratum $L_i$-line $l_i$ on $D_i$ such that $\mathscr{N}_{l_i/D_i}$ is semiample. If some line $l_i$ is free on two components of $\Delta$ then by Lemma \ref{lemma-two-free-lines} it is free on $X$, and we are done. Hence, we may assume that each line $l_i$ is free on exactly one component of $\Delta$. In particular, each of $n$  stratum $L$-lines on $X$ is free on some component of $\Delta$. 

Fix a polarization $L=-K_X-\Delta$ on $X$. By Section \ref{section_adjunction_theory} there exists a nef value morphism $\psi$ on $X$ for the nef value $\tau$ with respect to $L$. Consider an extremal ray $R$ in the extremal face contracted by $\psi$. Then there exists an extremal $(K_X + \tau L)$-trivial contraction $\phi=\phi_R$. Let $C$ be an effective curve in the extremal ray $R$. We consider two cases: $\phi$ is birational and $\phi$ is of fiber type.\\

\textbf{Case 1.}
$\phi$ is birational. We will derive a contradiction by showing that $\phi$ contracts at least two stratum $L$-lines on $X$. Since their deformations cover at least two of the components of $\Delta$, it follows that $\phi$ must contract at least two components of $\Delta$. This contradicts the fact that the contraction $\phi=\phi_R$ is birational and extremal. More precisely, in this case $\phi$ has relative Picard number $1$ and hence its exceptional locus cannot contain two different divisors. 

So we prove that at least two stratum $L$-lines are contracted by $\phi$. By Proposition \ref{inequalities}(\ref{ineq-3}) we have $\Delta \cdot C > 0$. Possibly after reordering of the components of $\Delta$, we may assume that $D_1\cdot C > 0$ and thus the exceptional locus of $\phi$ intersects $D_1$. By Proposition \ref{inequalities}(\ref{ineq-4}) one has $\dim F\geq 2$, where $F$ is a general non-trivial fiber of $\phi$. Hence $\dim (F\cap D_1)\geq 1$. Observe that $\phi$ induces a non-trivial contraction $\phi_{D_1}$ on $D_1$ (if $\phi|_{D_1}$ induces a contraction of an extremal face of positive dimension, we pass to a contraction of an extremal ray in this face and call it $\phi_{D_1}$). By Theorem \ref{the-theorem}(\ref{thm-stratum-set}) applied to the contraction $\phi_{D_1}$ on $(D_1, \Delta_{D_1})$, there exists a fiber $F$ of $\phi_{D_1}$ that coincides with some stratum of $\Delta_{D_1}$. 

In particular, there exists a stratum $L$-line $l$, such that $l\subset F \subset D_1$, contracted by $\phi_{D_1}$, and hence by $\phi$. Note that $l$ is proportional to $C$ on $X$, and hence $D_1\cdot l > 0$. By our assumption, $l$ is free on some component of $\Delta$. If it is free on $D_1$, then $\phi$ is divisorial and $\mathrm{Exc}\,(\phi)=D_1$. This contradicts the negativity lemma and the fact that $l \cdot D_1 > 0$. 

So $l$ is free on some other component of $\Delta$, say on $D_2$. Note that $l\subset D_1\cap D_2$. Thus $\phi$ is a divisorial contraction whose exceptional locus coincides with $D_2$. Consider an induced extremal contraction on $D_2$, call it $\phi_{D_2}$ (again, we pass to an extremal contraction that contracts $l$ if $\phi|_{D_2}$ contracts an extremal face of positive dimension). Since $l$ is free on $D_2$, we have that $\phi_{D_2}$ is of fiber type. Moreover, by Theorem \ref{the-theorem}(\ref{thm-projective-bundle}) it is a $\mathbb{P}^k$-bundle. If $k\geq 2$, then $\phi_{D_2}$ contracts at least two stratum $L_{D_2}$-lines. Hence, $\phi$ contracts at least two stratum $L$-lines on $X$ and we are done. Hence assume that $k=1$. By Theorem \ref{the-theorem}(\ref{thm-projective-bundle}), some component of $\Delta_{D_2}$ is a section of a $\mathbb{P}^1$-bundle $\phi_{D_2}$, while the other components are composed of fibers. Since $l\subset D_1\cap D_2$, the stratum $D_1\cap D_2$ cannot be a section, thus $D_1\cap D_2$ is a composed of fibers of $\phi_{D_2}$. Therefore $D_1\cdot l = 0$ which contradicts the fact that $l \cdot D_1 > 0$. This contradiction completes the proof of this case.\\

\textbf{Case 2.}
$\phi$ is a fibration. Then by Lemma \ref{lem-fiber-type}, we have that $\phi$ induces a contraction of fiber type on some $D_i$.
Let $\phi_{D_i}$ be an extremal contraction that corresponds to an extremal ray contracted by $\phi|_{D_i}$. By Theorem \ref{the-theorem}(\ref{thm-stratum-set}), we know that $\phi_{D_i}$ has a fiber $F$ that coincides with a stratum of $\Delta_{D_i}$. In particular, $F$ contains a stratum $L_{D_i}$-line $l$ on $D_i$. Since $l$ is contracted by $\phi$ which is of fiber type, the normal bundle $\mathscr{N}_{l/X}$ is semiample. $l$ is also a stratum $L$-line on $X$. Hence, we have found a stratum $L$-line on $X$ such that $\mathscr{N}_{l/X}$ is semiample. Thus, the proposition is proved.
\end{proof}

From now on, we assume that a free stratum $L$-line on $X$ is as follows: $l=D_2\cap \ldots \cap D_n$. The proof of the next proposition is based on the ideas introduced in \cite{BSW92}*{1.4.3, 1.4.5}.

\begin{proposition}
\label{prop-curve-equiv}
Let $C$ be an irreducible curve on $X$.
Then, there exist a constant $\alpha\in \rr$ and an effective $1$-cycle $C_{D_1}$ on $D_1$ such that the numerical equivalence
\[
C\equiv C_{D_1}+\alpha l
\]
holds.
In particular, $l$ and the curves contained in $D_1$ span $\mathrm{N}_1(X)$.
\end{proposition}
\begin{proof}
Due to Proposition \ref{prop-free-curve} there exists a free stratum $L$-line $l$ on $X$. Consider an arbitrary irreducible curve $C$ in $X$. We will show that $C$ is numerically equivalent to a linear combination of $l$ and an effective $1$-cycle $C_{D_1}$ contained in $D_1$. Clearly, we may assume that $C$ is not a deformation of $l$. Since $L\cdot l = 1$ and $D_1\cdot l = 1$, any curve which is a deformation of $l$ is irreducible, reduced, and intersects $D_1$.
Let $P$ be the Hilbert polynomial of $l$ with respect to the ample line bundle $L$.
Let $\mathcal{H}$ be the Hilbert scheme of subschemes of $X$ with Hilbert polynomial $P$ with respect to $L$.
Let $\mathcal{U}$ be the corresponding universal family (see, e.g.,~\cite{Kol96}*{Theorem 1.4}).
Consider the universal family morphism $u\colon \mathcal{U} \rightarrow X$.
Since $C$ is not a deformation of $l$, we can find an one-dimensional subscheme $C_{\mathcal{U}}$ of $\mathcal{U}$ mapping surjectively to $C$ so that $C_{\mathcal{U}}$ is not a fiber of $v: \mathcal{U}\rightarrow \mathcal{H}$.
Hence, the image of $C_{\mathcal{U}}$ on $\mathcal{H}$ is one-dimensional.
We can find a morphism from a smooth projective curve $C_0$
to $\mathcal{H}$ whose image is contained in $v(C_{\mathcal{U}})$.
Consider the fiber product
$S:= C_0\times_{\mathcal{H}}\mathcal{U}$.
By construction, all the fibers of $S\rightarrow C_0$ are isomorphic to $\pp^1$. Hence, $S$ is a ruled surface which admits a morphism 
$\psi \colon S\rightarrow X$.
Furthermore,
the image via $\psi$ of every fiber of $S\rightarrow C_0$ intersects $C$.

By construction, the morphism $\psi$ maps a ruling $f$ of $S$ to a curve $l'$ which intersects both $C$ and $D_1$. Moreover, $l'$ is a deformation of $l$.
The image $\psi(S)$ intersects $D_1$ in a curve $C_{D_1}$.
By construction, $\psi(S)$ contains the curve $C$.
We can find effective $1$-cycles
$C_{D_1,S}\subset S$ and $C_S\subset S$ so that
\[
\psi_* C_{D_1,S} = C_{D_1} 
\text{ and }
\psi_*C_S = C.
\]
Here, the push-forward is the classic push-forward of 1-cycles.
We claim that $C_{D_1,S}$ has a horizontal component over $C_0$.
Otherwise, the image on $X$ of a general fiber of $S\rightarrow C_0$ is disjoint from $D_1$, leading to a contradiction.
Thus, we can write
\[
C_S\equiv C_{D_1,S}+\alpha f
\]
for some $\alpha\in \rr$, where
$f$ is a ruling of $S\rightarrow C_0$.
We claim that the equivalence
\[
C \equiv C_{D_1}+ \alpha  l
\]
holds on $X$.
It suffices to prove that
\[
C \cdot \mathcal{L} = (C_{D_1}+\alpha l)\cdot \mathcal{L}
\]
for any line bundle $\mathcal{L}$ on $X$.
Indeed, the above equivalence follows from 
\[
C_S\cdot \psi^*\mathcal{L}
= 
(C_{D_1,S}+\alpha f)\cdot \psi^*\mathcal{L}
\]
and the projection formula. This proves the proposition.
\end{proof}
\begin{corollary}
\label{cor-picard-bound}
We have $\rho(D_1) \leq \rho(X)\leq \rho(D_1)+1$. As a consequence, $\rho(X)\leq n$.
There are two cases:
\begin{enumerate}
\item
$\rho(X)=\rho(D_1)+1$ and $\mathrm{N}_1(X)\simeq \mathrm{N}_1(D_1)\oplus  \langle l \rangle$,
\item
$\rho(X)=\rho(D_1)$ and $\mathrm{N}_1(X)\simeq \mathrm{N}_1(D_1)$.
\end{enumerate}
\end{corollary}
\begin{proof}
Consider the map $\Pic (X)\to \Pic (D_1)$ given by the restriction $D\mapsto D|_{D_1}$. This map is surjective. Indeed, by Theorem \ref{the-theorem}(\ref{thm-picard}) applied to $D_1$, the divisors $D_i|_{D_1}$ for $i\geq 2$ generate $\Pic (D_1)$. Therefore $\rho(X)\geq \rho (D_1)$. Note that the induced map $\mathrm{N}_1(D_1)\to \mathrm{N}_1(X)$ is injective. 

By Theorem \ref{the-theorem}(\ref{thm-picard}) applied to $D_1$, one has $\rho(D_1)\leq n-1$. Hence it is enough to prove that $\rho(X)\leq \rho(D_1)+1$, or by duality, that $\dim \mathrm{N}_1(X)\leq \dim \mathrm{N}_1(D_1)+1$. This follows from Proposition \ref{prop-curve-equiv}. Moreover, we see that $\mathrm{N}_1(X)$ is generated by $\mathrm{N}_1(D_1)$ and the line $l$. Since $l$ either belongs to $\mathrm{N}_1(D_1)$ considered as a subspace of $\mathrm{N}_1(X)$, or does not belong to it, we have the two desired cases.
\end{proof}

In Sections \ref{sect-first-case} and \ref{sect-second-case} we treat two cases considered in Corollary \ref{cor-picard-bound} separately.

\section{\texorpdfstring{The case $\rho(X)=\rho(D_1)+1$}{The case rho(X)=rho(D1)+1}}
\label{sect-first-case}
We work in the notation of Sections \ref{section-geometry} and \ref{sect-looking-for-a-line}. In this section, we deal with the case  
\begin{equation}
\label{eq-group-of-curves}
\begin{split}
\rho(X)=\rho(D_1)+1, \quad \mathrm{N}_1(X)\simeq \mathrm{N}_1(D_1)\oplus  \langle l \rangle,
\end{split}
\end{equation} 
see Corollary \ref{cor-picard-bound}. We assume this throughout the section. The curve $l$ is as defined in Proposition \ref{prop-free-curve}, and $D_1$ is a unique component of the boundary $\Delta$ on $X$ such that $l\not\subset D_1$.

\begin{proposition}
\label{prop-l-extremal}
\label{prop-p1-bundle}
The curve $l$ is extremal in $\overline{\mathrm{NE}}(X)$. The contraction $\pi$, corresponding to the extremal curve $l$, induces a $\mathbb{P}^1$-bundle structure on $X$ such that $\pi(X)$ is isomorphic to $D_1$ and $D_1$ is a section of $\pi$.
\end{proposition}
\begin{proof}
We start with the first assertion. Suppose that we have $l \equiv C + C'$ where $C$ and $C'$ are effective $1$-cycles on $X$. By Proposition \ref{prop-curve-equiv}, we have 
\[C\equiv C_{D_1} + \alpha l \text{ and } C'\equiv C'_{D_1} + \alpha' l,\]
where $C_{D_1}$ and $C'_{D_1}$ are effective $1$-cycles on $D_1$ and $\alpha, \alpha' \in \mathbb{R}$. Since for any $1$-cycle $\gamma$ in $D_1$, $\gamma$ and $l$ are linearly independent in $\mathrm{N}_1(X)$, we have $C_{D_1} + C'_{D_1}\equiv 0$ and hence $C_{D_1}\equiv C'_{D_1}\equiv 0$. Thus $C$ and $C'$ are numerically proportional to $l$ which means that $l$ is extremal in $\overline{\mathrm{NE}}(X)$. 

We prove the second assertion. Since $l$ is extremal, there exists an extremal contraction of $l$ on $X$ that we denote by $\pi$. Since $l$ is free, $\pi$ is $K_X$-negative and of fiber type. We claim that $\pi|_{D_1}$ is an isomorphism. Indeed, since $\pi$ contracts precisely the curves numerically equivalent to $l$ and (\ref{eq-group-of-curves}) holds, $\pi$ does not contract any curve on $D_1$. Thus, $\pi$ induces an isomorphism on $\overline{\mathrm{NE}}(D_1)$. It follows that any fiber of $\pi$ is $1$-dimensional. Note that since the general fiber is smooth, it is isomorphic to $\mathbb{P}^1$. Since $l \cdot D_1 = 1$, the morphism $D_1\to \pi(D_1)$ is birational. Since the image $\pi(D_1)$ is normal, so it follows that $\pi|_{D_1}$ is an isomorphism. Hence $D_1$ is a section of $\pi$. By Theorem \ref{thm-charact-scroll}, we conclude that $X$ is isomorphic to a $\mathbb{P}^1$-bundle over $\pi(X)\simeq D_1$. 
\end{proof} 

Consider the contraction $\pi$ obtained in the proof of Proposition \ref{prop-p1-bundle}. Since it is a $\mathbb{P}^1$-bundle, we have that $(-K_X)\cdot l = 2$. Since $L\cdot l = (-K_X - \Delta)\cdot l = 1$, we have that
\[
\Delta\cdot l = 1,\ D_1\cdot l = 1\ \text{ and }\ D_2 \cdot l = \ldots = D_n \cdot l = 0.
\]
Therefore the components $D_2, \ldots, D_n$ are composed of fibers of $\pi$. In fact, using Theorem \ref{the-theorem} applied to $D_i$, one checks that each induced contraction $\pi|_{D_i}: D_i \to D_1\cap D_i$ is a $\mathbb{P}^1$-bundle for any $i\geq 2$.

\begin{proposition}
\label{prop-stratum-lines}
The Mori cone of $X$ is generated by stratum $L$-lines.
\end{proposition}
\begin{proof}
Suppose the contrary. Then, since $\overline{\mathrm{NE}}(X)$ is polyhedral, there exists an extremal effective curve $C$ which is not numerically equivalent to a non-negative linear combination of stratum $L$-lines on $X$. We will derive a contradiction. Let $\phi$ be a $(K_X+\Delta)$-negative contraction of an extremal ray $R$ corresponding to $C$, so $\phi(C)$ is a point. We may assume that $C$ is irreducible, and that $l(R)$ is attained at $C$, that is $l(R)=-K_X\cdot C$.

Assume that $\phi$ is of fiber type. By Lemma \ref{lem-fiber-type}, $\phi$ induces a contraction of fiber type on some $D_i$. Then, an effective curve $C'$, numerically proportional to $C$, is contained in some $D_i$. This leads to a contradiction. Indeed, by Theorem \ref{the-theorem}(\ref{thm-stratum-lines}) applied to $D_i$, the curve $C'$, hence $C$, is numerically equivalent to a non-negative linear combination of stratum $L$-lines on $D_i$, hence on $X$. 

Assume that $\phi$ is birational. First, we show that $D_i \cdot C\geq 0$ for any $i$. Indeed, if not, then $C\subset D_i$ for some $i$. Then, by Theorem \ref{the-theorem}(\ref{thm-stratum-lines}) applied to $D_i$, the curve $C$ is equivalent to a non-negative linear  combination of stratum lines on $D_i$, and hence on $X$. Thus $D_i \cdot C\geq 0$ for any $i$ and $\Delta\cdot C\geq 0$. We claim that $\phi$ is $K_X$-negative. Indeed, since $(-K_X-\Delta)\cdot C>0$, we have $K_X\cdot C < 0$. Consider a (possibly non-normal) ruled surface $S$ that is composed of fibers of $\pi$ that intersect $C$. Let $C_{D_1}$ be the intersection of $S$ with $D_1$. By Theorem \ref{the-theorem}(\ref{thm-picard}) applied to $D_1$, the divisors $D_i|_{D_1}$ for $i\geq 2$ generate the Picard group of $D_1$. Therefore $C_{D_1}$ intersects one of the $D_i|_{D_1}$'s, say $D_2|_{D_1}$. Above, we have seen that $D_i$ for $i\geq 2$  are composed of fibers of $\pi$ and hence $C$ intersects $D_2$. Then $\Delta \cdot C \geq 1$ and $(-K_X)\cdot C \geq 2$. By Lemma \ref{inequalities}(\ref{ineq-4}) any non-trivial fiber of $\phi$ is of dimension $\geq 2$. Hence, there exists an effective curve $C'$ contracted by $\phi$ that is contained in $D_2$. By Theorem \ref{the-theorem}(\ref{thm-stratum-lines}) applied to $D_2$ the curve $C'$ is a non-negative linear combination of  stratum lines on $D_2$, and hence on $X$. Since $C'$ is proportional to $C$ on $X$, we get a contradiction.
\end{proof}

Our next goal is to show that the $\mathbb{P}^1$-bundle given by the contraction $\pi$ is a projectivization of a split vector bundle $\mathscr{E}$. To this aim, we construct a $1$-complement $\Gamma$ of $K_X+\Delta$ whose horizontal component $\Gamma_1$ is a section of $\pi$ disjoint from $D_1$. Two disjoint sections will give us a splitting of $\mathscr{E}$.

\begin{example}{\em 
The condition that a log Fano manifold $(X, \Delta)$ is maximal is important for the existence of a splitting of $\mathscr{E}$. Indeed, take $X$ to  be $\mathbb{P}_{\mathbb{P}^2}(T_{\mathbb{P}^2})$, that is, the projectivization of the tangent bundle on $\mathbb{P}^2$. It is a threefold Fano variety of index $2$ usually denoted by $V_6$. Alternatively, $X$ can be described as a hypersurface of bidegree $(1, 1)$ in $\mathbb{P}^2\times \mathbb{P}^2$. We have $\rho(X)=2$ and there are two extremal contractions $\pi_1$ and $\pi_2$ induced by projections to both copies of $\mathbb{P}^2$. Put $\Delta = D_1 + D_2$ where $D_1$ is a divisor of bidegree $(1, 0)$, and $D_2$ is a divisor of bidegree $(0, 1)$. One checks that $-K_X - \Delta$ is ample, so $(X, \Delta)$ is a log Fano manifold, but it is not maximal. Note that each $D_i$ is isomorphic to $\mathbb{F}_1$. Moreover, $D_1$ intersects $D_2$ in a smooth rational curve which is equivalent to $s + f$ on each $D_i\simeq \mathbb{F}_1$, where $s$ is a $(-1)$-curve and $f$ is a fiber. Then $D_1$ is a birational section of $\pi_2$ and $D_2$ is composed of fibers of $\pi_2$. The contraction $\pi_2$ is a projection $X = \mathbb{P}_{\mathbb{P}^2}(T_{\mathbb{P}^2}) \to \mathbb{P}^2$ which is a non-split $\mathbb{P}^1$-bundle.}
\end{example}

Now we show that we can lift the complement $\Gamma_\Delta$ constructed in Lemma \ref{lemma-glue-gamma} from $\Delta$ to $X$. Our aim is to show that $X$ is toric and $\Delta + \Gamma$ is the toric boundary. However, we need to be careful in choosing a lift of $\Gamma_\Delta$ as the following example shows.

\begin{example}{\em 
\label{rem-toric-compl}{ 
Consider a two-dimensional maximal log Fano manifold $(\mathbb{P}^1\times\mathbb{P}^1, \Delta =  l_1+ l_2)$, where $l_1$ has bidegree $(1, 0)$ and $l_2$ has bidegree $(0, 1)$. Then there exists an irreducible $1$-complement $\Gamma$ of $K_{\mathbb{P}^1\times\mathbb{P}^1} + \Delta$, namely a line of bidegree $(1, 1)$, which is not toric. However, if we take a reducible $1$-complement $\Gamma=\Gamma_1+\Gamma_2$, $\Gamma_1\sim l_1$ and $\Gamma_2\sim l_2$, then it is toric. Both complements are lifts of $\Gamma_\Delta$ on $\Delta = l_1 + l_2$ which means that $\Gamma|_\Delta = \Gamma_\Delta$.}
}
\end{example}

\begin{proposition}
\label{prop-log-smooth-complement}
There exists an $1$-complement $\Gamma_1$ of the divisor $K_X+\Delta+\pi^* \Gamma_{D_1}$ on $X$ where $\Gamma_{D_1}$ is a toric $1$-complement on $D_1$. Moreover, the divisor $\Gamma = \pi^* \Gamma_{D_1} + \Gamma_1$ is log smooth, $\Gamma_1$ is irreducible, and $\Gamma$ has $\rho(X)$ components.
\end{proposition}
\begin{proof}
Consider an exact sequence
\[
0 \to \oo_X ( -K_X - 2\Delta - \pi^* \Gamma_{D_1} ) \to \oo_X ( -K_X - \Delta - \pi^* \Gamma_{D_1} ) \to \oo_{\Delta} ( -K_{\Delta} - \pi^* \Gamma_{D_1}|_\Delta) \to 0.
\]
By Theorem \ref{the-theorem} applied to $D_i$ for $i\geq 2$, each contraction $D_i \to D_1\cap D_i$ is a $\mathbb{P}^1$-bundle. Note that 
\[
\pi^*\Gamma_{D_1}|_{D_1} = \Gamma_{D_1}\  \ \text{and}\ \ \pi^*\Gamma_{D_1}|_{D_i} = \pi^* \Gamma_{D_1\cap D_i}\ \ \text{for}\ \ i\geq 2.,
\] 
So $\pi^*\Gamma_{D_1}$ restricted to each component $D_i$ of $\Delta$ is a torus-invariant divisor whose components are not contained in $\Delta_{D_i}$. Hence $\Gamma_\Delta - \pi^*\Gamma_{D_1}|_\Delta$ is an effective integral divisor on $\Delta$, where $\Gamma_{\Delta}$ is the $1$-complement constructed in Lemma \ref{lemma-glue-gamma}. Moreover, 
\[
\Gamma_\Delta - \pi^*\Gamma_{D_1}|_\Delta \sim -K_{\Delta} - \pi^* \Gamma_{D_1}|_\Delta.
\]
We want to lift this divisor from $\Delta$ to $X$. To do this, it is enough to show that
\[
\mathrm{H}^1(X,\mathcal{O}_X(-K_X - 2\Delta - \pi^* \Gamma_{D_1})) \simeq \mathrm{H}^1(X,\mathcal{O}_X(K_X + 2 ( - K_X - \Delta) - \pi^* \Gamma_{D_1})) =0.
\] 
To apply Kawamata-Viehweg vanishing theorem, we need to check that $- 2 ( K_X + \Delta ) - \pi^* \Gamma_{D_1}$ is ample. To this aim, it is enough to show that its intersection with any stratum $L$-line on $X$ is positive. Restricted to $D_1$ this divisor is just $-K_{D_1}-\Delta_{D_1}$ which is ample. Intersecting it with $l$ we get $2$, hence we are done. 

Thus, there exists an $1$-complement $\Gamma_1$ of $K_X+\Delta+\pi^*\Gamma_{D_1}$ which restricts to $\Gamma_\Delta - \pi^*\Gamma_{D_1}|_\Delta$ on $\Delta$. Put $\Gamma=\pi^*\Gamma_{D_1}+\Gamma_1$. Restricting to a fiber of $\pi$, we see that $\Gamma_1$ has exactly one $\pi$-horizontal component. From the formula 
\[
0 \sim K_X+\Delta + \Gamma = K_X + \Delta + \pi^* \Gamma_{D_1} + \Gamma_1
\] 
restricted to $D_1$ it follows that $\Gamma_1|_{D_1} = 0$. Hence, $\Gamma_1$ does not have $\pi$-vertical components. Since $\Gamma_1\cdot l = 1$, it is a section of $\pi$. Therefore $\Gamma_1\simeq D_1$, so $\Gamma_1$ is irreducible and smooth. Moreover, by \ref{the-theorem}(\ref{thm-gamma}) applied to $D_1$, the divisor $\pi^*\Gamma_{D_1}$ is log smooth. Since $\Gamma_1\cap D_1 = \emptyset$, the divisor $\Gamma = \sum \Gamma_i$ is log smooth. 

Now, we show that $\Gamma$ has $\rho(X)$ components. Since by Theorem \ref{the-theorem}(\ref{thm-gamma}) the complement $\Gamma_{D_1}$ has $\rho(D_1)$ components, $\Gamma = \Gamma_1 + \pi^*\Gamma_{D_1}$ and $\Gamma_1$ is irreducible. It follows that $\Gamma$ has exactly $\rho(X)$ components. The proposition is proved.
\end{proof}

\begin{proposition}
\label{prop-toric-pair}
The pair $(X, \Delta)$ is toric.
\end{proposition}
\begin{proof}
The base $\pi(X)\simeq D_1$ of the $\mathbb{P}^1$-bundle $\pi$ is toric by Theorem \ref{the-theorem}(\ref{thm-toric-pair}). Moreover, $\pi$ has two disjoint sections $\Gamma_1$ and $D_1$. Hence $X$ is isomorphic to a projectivization $\mathbb{P}_{D_1}(\mathscr{E})$ of a split vector bundle $\mathscr{E}$ over a toric variety $D_1$. We conclude that $X$ is toric, see \cite{CLS11}*{7.3.3}. Toricity of $\Delta$ follows from toricity of the horizontal component $D_1$, and the fact that $\Delta - D_1$ consists of $\pi$-vertical components of $\Delta$ whose restrictions to $D_1$ form a torus-invariant boundary divisor $\Delta_{D_1}$.
\end{proof}

In the next section, we will treat the second case of Corollary \ref{cor-picard-bound}. The proof of Theorem \ref{the-theorem} will be given in Section \ref{sect-proof-of-thm}.

\section{\texorpdfstring{The case $\rho(X)=\rho(D_1)$}{The case rho(X)=rho(D1)}}
\label{sect-second-case}

We work in the notation of Sections \ref{section-geometry} and \ref{sect-looking-for-a-line}. In this section we consider the case 
\[
\rho(X)=\rho(D_1) \text{ and } \mathrm{N}_1(X)\simeq \mathrm{N}_1(D_1)
\] 
as in Corollary \ref{cor-picard-bound}. We assume this throughout the section. By Theorem \ref{the-theorem}(\ref{thm-stratum-lines}) the group $\mathrm{N}_1(D_1)$, and hence $\mathrm{N}_1(X)$, is generated by stratum $L$-lines on $D_1$. In Corollary \ref{cor-picard-bound}, we have shown that the map $\mathrm{N}^1 (X) \to \mathrm{N}^1 (D_1)$ given by the restriction $D\mapsto D|_{D_1}$ is surjective, and since $\rho(X)=\rho(D_1)$, it is an isomorphism. The inclusion map $\mathrm{N_1}(D_1)\to \mathrm{N_1}(X)$ is an isomorphism as well. Hence $\overline{\mathrm{NE}}(D_1)$ embeds into $\overline{\mathrm{NE}}(X)$. We will show that these Mori cones coincide. First, we will prove the following statement.

\begin{proposition}
\label{prop-l-equiv}
There exists a free stratum $L_{D_1}$-line $l_{D_1}$ on $D_1$ such that $l_{D_1}\equiv l$ on $X$, where $l$ is as in Proposition \ref{prop-free-curve}.
\end{proposition}
\begin{proof}
By our assumption, $l=D_2\cap\ldots\cap D_n$. Since $l$ is free on $X$, it is also free on $D_i$ for any $i \geq 2$, see decomposition \eqref{eq-normal-bundles}. By Theorem \ref{the-theorem}(\ref{thm-projective-bundle}) the curve $l$ is extremal on $D_i$ and there exists an extremal contraction $\pi_i$ of $l$ which is of fiber type on $D_i$ for any $i\geq 2$. By Theorem \ref{the-theorem}(\ref{thm-projective-bundle}) each contraction $\pi_i$ is a $\mathbb{P}^{k_i}$-bundle for some $k_i\geq 1$. Consider two cases: there exists $i$ such that $k_i\geq 2$, and $k_i=1$ for any $i\geq 2$. 

Assume $k_i\geq 2$ for some $i$. Since $l \cdot D_1 > 0$, every fiber of $\pi_i$ intersects $D_1$ in a positive-dimensional subvariety. By Theorem \ref{the-theorem}(\ref{thm-stratum-set}), some fiber $F$ of $\pi_i$ coincides with a stratum of $\Delta_{D_i}$. Hence $F\cap D_1$ contains some stratum of $\Delta_{D_1}$. In particular, this intersection contains some stratum $L_{D_1}$-line which we call $l_{D_1}$. Since both $l$ and $l_{D_1}$ are contracted by $\pi_i$ and $\pi_i$ is extremal, these curves are numerically proportional on $D_1$, and hence on $X$. But since both of them are $L$-lines, we have $l_{D_1}\equiv l$, and the proposition is proved.

Now, we assume $k_i=1$ for any $i\geq 2$, so each $\pi_i$ is a $\mathbb{P}^1$-bundle. 
Therefore $\mathscr{N}_{l/D_i}$ is trivial for $2\leq i\leq n$. Then from decomposition \eqref{eq-normal-bundles} it follows that $\mathscr{N}_{l/X}$ is trivial. Hence we have
\[
(-K_X)\cdot l = 2, \quad D_1\cdot l = 1, \quad D_2\cdot l = \ldots = D_n\cdot l = 0.
\] 
By Theorem \ref{the-theorem}(\ref{thm-picard}), we know that $\mathrm{N}^1 (D_1)$ is generated by restrictions of $D_i|_{D_1}$ for $i\geq 2$. 
We conclude that $\mathrm{N}^1 (X)$ is generated by $D_2, \ldots, D_n$. But then $l = 0$ on $X$. This contradiction concludes the proof.
\end{proof}

\begin{proposition}
\label{prop-stratum-lines2}
We have $\overline{\mathrm{NE}}(X)\simeq \overline{\mathrm{NE}}(D_1)$. As a consequence, the Mori cone of $X$ is generated by stratum $L$-lines.
\end{proposition}
\begin{proof}
Suppose the opposite. Then there exists an extremal effective curve $C$ which is not a non-negative linear combination of stratum $L$-lines on $X$. We will derive a contradiction. Let $\phi$ be a $(K_X+\Delta)$-negative contraction of an extremal ray $R$ such that the class of $C$ belongs to $R$. Then $\phi$ induces a map with a non-trivial kernel on $\mathrm{N}_1(X)\simeq \mathrm{N}_1(D_1)$. Thus, there exists an $1$-cycle $Z$ on $D_1$ that is mapped to a zero cycle by $\phi$. We may assume that no effective curve on $D_1$ is contracted by $\phi$. Indeed, otherwise by Theorem \ref{the-theorem}(\ref{thm-stratum-lines}) applied to $D_1$, such a curve would be a non-negative linear combination of stratum lines on $D_1$ (and hence on $X$), which gives a contradiction. We see that some fibers of $\phi$ intersect $D_1$, so $D_1\cdot C > 0$. 

If $\phi$ is of fiber type, the same argument given in the proof of Proposition \ref{prop-stratum-lines} applies. So, we assume that $\phi$ is birational. Then as in the proof of Proposition \ref{prop-stratum-lines}, we may assume that $D_i\cdot C\geq 0$ for any $i$ and hence $\Delta\cdot C \geq 0$. Thus, $\phi$ is $K_X$-negative. We have shown that some curve numerically equivalent to $C$ intersects $D_1$, and hence $\Delta \cdot C\geq 1$. We conclude that $l(R)\geq 2$. By Proposition \ref{inequalities}(\ref{ineq-4}), non-trivial fibers of $\phi$ have dimension $\geq 2$. Since some fiber of $\phi$ intersects $D_1$, there exists a curve in $D_1$ contracted by $\phi$. This contradicts our assumption, because such curve would be a non-negative linear combination of stratum lines on $D_1$ and hence on $X$. The proof is complete.   
\end{proof}

By Proposition \ref{prop-l-equiv}, we have that $l\equiv l_{D_1}$ for some stratum $L_{D_1}$-line $l_{D_1}$ on $D_1$. By Theorem \ref{the-theorem}(\ref{thm-stratum-lines}), the curve $l_{D_i}$ is extremal in $\overline{\mathrm{NE}}(D_1) \simeq \overline{\mathrm{NE}}(X)$. So there exists an extremal contraction $\pi$ of fiber type on $X$ such that it contracts precisely the curves numerically equivalent to $l$. Note that $\pi$ induces an extremal contraction of fiber type $\pi_{D_1} = \pi|_{D_1}$ on $D_1$. 

By Theorem \ref{the-theorem}(\ref{thm-projective-bundle}), we know that any fiber of $\pi_{D_1}$ is isomorphic to $\mathbb{P}^{k-1}$ for some $k\geq 2$. Moreover, possibly after reordering, $D_2|_{D_1}, \ldots, D_{k}|_{D_1}$ are horizontal for $\pi_{D_1}$, and $D_2|_{D_1} \cap \ldots \cap D_{k}|_{D_1}$ is a section of $\pi_{D_1}$. Therefore $D_1, D_2, \ldots, D_{k}$ are horizontal for $\pi$. 

\begin{proposition}
\label{prop-projective-bundle2}
The contraction $\pi$ on $X$ is a $\mathbb{P}^{k}$-bundle where $k\geq 2$ is as above. $\pi(X)\simeq D$ where $D=D_1\cap \ldots \cap D_{k}$. Furthermore,  $D$ is a section of $\pi$.
\end{proposition}
\begin{proof}
We claim that any fiber of $\pi$ has dimension $k$. Indeed, since $D_1\cdot l=1$ and $\pi$ is extremal, $D_1$ is $\pi$-ample, hence any component of any fiber intersects $D_1$. By the above discussion, the intersection of any fiber to $D_1$ is isomorphic to $\mathbb{P}^{k-1}$. This implies that the fibers of $\pi$ have dimension $k$. 

Now, we prove the claim of the proposition for a general fiber $F$ of $\pi$. Clearly, thie fiber is smooth. We have $D_1\cdot l = 1$ and by the above discussion $D_2 \cdot l = \ldots = D_{k} \cdot l = 1$. Using the fact that $K_X|_F = K_F$, by the adjunction formula, we obtain 
\[
K_X+\Delta|_F = K_F + \Delta_F
\] 
where $\Delta_F$ has at least $k$ components. Moreover, $-K_F - \Delta_F$ is ample on $F$. Since the contraction is extremal, all the components of $\Delta_k$ are proportional to $K_F$. Hence $F$ is a Fano manifold of index $k+1$. Since $F$ has dimension $k$, by Theorem~\ref{thm-index-charact}, we have $F\simeq \mathbb{P}^k$. Since we know that any fiber of $\pi$ has dimension $k$, we can apply Theorem \ref{thm-charact-scroll} to conclude that every fiber of $\pi$ is isomorphic to $\mathbb{P}^{k}$ and $\pi$ is a $\mathbb{P}^{k}$-bundle. Since each $D_i$ for $1\leq i\leq k$ is a hyperplane on $F$, we conclude that $D=D_1\cap \ldots \cap D_{k}$ is a section of $\pi$.
\end{proof}

One checks that $\pi|_{D_i}$ for $1\leq i \leq k$ is a $\mathbb{P}^{k-1}$-bundle. Moreover, $X$ is isomorphic to $\mathbb{P}_D(\mathscr{E})$ where the sheaf $\mathscr{E}=\pi_*\oo_X(D_1)$ is locally free. It follows that $\rho(X)=\rho(D)+1$. Now, we show that the vector bundle $\mathscr{E}$ splits. To this aim, we will construct a complement $\Gamma$ of $K_X + \Delta$.

\begin{proposition}
\label{prop-log-smooth-complement2}
There exists an $1$-complement $\Gamma_1$ of the divisor $K_X+\Delta+\pi^* \Gamma_{D}$ on $X$ where $\Gamma_{D}$ is a toric $1$-complement on $D$. Moreover, the divisor $\Gamma = \pi^* \Gamma_{D} + \Gamma_1$ is log smooth, $\Gamma_1$ is irreducible, and $\Gamma$ has $\rho(X)$ components.
\end{proposition}
\begin{proof}
Repeating the argument given in the proof of Proposition \ref{prop-log-smooth-complement} for $D$ (instead of $D_1$), we get an $1$-complement of $K_X+\Delta+\pi^*\Gamma_{D}$ that restricts to $\Gamma_\Delta + \pi^*\Gamma_D|_\Delta$ on $\Delta$. We denote it by $\Gamma_1$ and put $\Gamma=\pi^*\Gamma_{D}+\Gamma_1$. Restricting to a fiber of $\pi$, we see that $\Gamma_1$ has exactly one $\pi$-horizontal component. From the formula 
\[
0 \sim K_X+\Delta + \Gamma = K_X + \Delta + \pi^* \Gamma_{D} + \Gamma_1
\] 
restricted to $D$, it follows that $\Gamma_1|_{D} = 0$. In other words, on each fiber $F$ of $\pi$ the intersection of the hyperplanes 
\[
D_1|_{F},\ldots, D_k|_{F}, \Gamma_1|_{F}
\] 
is empty. It follows $\pi$ induces on $\Gamma_1$ the structure of a $\mathbb{P}^{k-1}$-bundle over $D$. In particular, $\Gamma_1$ is irreducible and smooth. Moreover, by Theorem~\ref{the-theorem}(\ref{thm-gamma}), we know that $\pi^*\Gamma_{D}$ is log smooth. It follows that the divisor $\Gamma = \sum \Gamma_i$ is log smooth. 

Now, we show that $\Gamma$ has $\rho(X)$ components. By Theorem \ref{the-theorem}(\ref{thm-gamma}), the complement $\Gamma_{D}$ has $\rho(D)$ components, $\Gamma = \Gamma_1 + \pi^*\Gamma_{D_1}$, $\rho(X)=\rho(D)+1$, and $\Gamma_1$ is irreducible. It follows that $\Gamma$ has exactly $\rho(X)$ components. The proposition is proved.
\end{proof}

\begin{proposition}
\label{prop-toric-pair2}
$X$ is isomorphic to the projectivization of a split vector bundle on $D = D_1\cap \ldots \cap D_{k}$. As a consequence, the pair $(X, \Delta)$ is toric.
\end{proposition}

\begin{proof}
We have that $X \simeq \mathbb{P}_{D}(\mathscr{E})$ where $\mathscr{E}=\oo_D(\pi_*D_1)$ is a vector bundle on $D$ of rank $k+1$. We prove that $\mathscr{E}$ splits. By Proposition \ref{prop-log-smooth-complement2}, the complement $\Gamma$ on $X$ has a horizontal component $\Gamma_1$. Hence, on each fiber $F$ of $\pi$ there are $k+1$ hyperplanes $D_1|_F, \ldots, D_{k}|_F, \Gamma_1|_F$ such that their intersection is empty. Taking intersections of $k$ of them, we obtain $k+1$ sections of $\pi$ that are linearly independent over each point of $D$. This give us a splitting of $\mathscr{E}$ into a direct sum of line bundles. Then $(X, \Delta)$ is toric by \cite{CLS11}*{7.3.3}.
\end{proof}

\section{Proof of Theorem \ref{the-theorem}}
\label{sect-proof-of-thm}
\begin{proposition}
\label{prop-stratum-set}
Let $\phi$ be an extremal contraction on $X$. Then $\mathrm{Exc}\,(\phi)$ coincides with some stratum of $\Delta$ and there exists a fiber $F$ of $\phi$ that coincides with some stratum of $\Delta$.
\end{proposition}
\begin{proof}
At this moment, Theorem \ref{the-theorem}(\ref{thm-toric-pair}), (\ref{thm-gamma}), (\ref{thm-stratum-lines}), (\ref{thm-free-curve}), (\ref{thm-projective-bundle}), (\ref{thm-split-bundle}) is proved for $X$, see Sections \ref{sect-looking-for-a-line}, \ref{sect-first-case} and \ref{sect-second-case}. Hence, we will use the corresponding statements.

Assume that $\phi$ is birational. By Theorem \ref{the-theorem}(\ref{thm-stratum-lines}), we know that $\phi$ coincides with a contraction of a stratum $L$-line which we denote by $l$. Consider the largest stratum $D$ of $\Delta$ for which $\mathscr{N}_{l/D}$ is semiample, see decomposition \eqref{eq-normal-bundles}. We have that  $l\subset D$, so $D$ is non-empty. Since $\phi$ is birational, $D\neq X$, so $D\subset \Delta$. By Theorem \ref{the-theorem}(\ref{thm-projective-bundle}), we have $D\subset\mathrm{Exc}\,(\phi)$. Let us prove that these two subsets coincide. Indeed, $X$ is toric and smooth and $\phi$ is extremal. By \cite{Re83}*{2.5}, we know that $\mathrm{Exc}\,(\phi)$ is irreducible. Moreover, $\mathrm{Exc}\,(\phi)$ is contained in the support of the toric boundary $\Delta + \Gamma$. Since $\mathrm{Exc}\,(\phi)$ is irreducible, it is contained either in $\Delta$ or in $\Gamma$. By construction of $\Gamma$, $\Gamma\cap D = \Gamma_D$ which is a divisor on $\Gamma$, so $\mathrm{Exc}\,(\phi)$ is not contained in $\Gamma$. Thus $\mathrm{Exc}\,(\phi)\subset \Delta$. Again, since $\mathrm{Exc}\,(\phi)$ is irreducible, it is contained in some component $D_i$ of $\Delta$. We may apply Theorem \ref{the-theorem}(\ref{thm-stratum-set}) to $D_i$ (note that $\phi_{D_i}$ is extremal since it is a contraction of a stratum $L$-line $l$ on $D_i$) to conclude that $\mathrm{Exc}\,(\phi|_{D_i})=D$ on $D_i$. Hence $\mathrm{Exc}\,(\phi)=D$. By Theorem \ref{the-theorem}(\ref{thm-stratum-set}) applied to $D$, there is a fiber $F$ of $\phi|_{D}$ that coincides with some stratum of $\Delta_{D}$, and hence of $\Delta$. Hence $F$ is also a fiber of $\phi$, and the proposition is proved in the case when $\phi$ is birational.

Now, assume that $\phi$ is of fiber type. Then its exceptional locus coincides with $X$ which is a stratum subset, so we only have to find a fiber $F$ that coincides with some stratum of $\Delta$. By Theorem \ref{the-theorem}(\ref{thm-projective-bundle}), we see that $\phi$ is a $\mathbb{P}^k$-bundle for some $1 \leq k \leq n$. If $k=n$, then the whole $X$ is a fiber $F$ which is a stratum subset. So assume that $k<n$. Then as shown above the image $\phi(X)$ is isomorphic to some stratum $D=D_1\cap\ldots \cap D_k$, and $\Delta$ has at least one component, say $D_n$, composed of fibers of $\phi$. Apply Theorem \ref{the-theorem}(\ref{thm-stratum-set}) to $D_n$ to conclude that $\phi_{D_n}$, and hence $\phi$, has a fiber $F$ that coincides with a stratum of $\Delta$.
\end{proof}

\begin{proposition}
\label{prop-nef-component}
There exists a component $D_i$ of $\Delta$ such that $|D_i|$ is basepoint free. In particular, $D_i$ is nef on $X$, and $\mathscr{N}_{D_i/X} = \oo_{D_i}(D_i|_{D_i})$ is nef on $D_i$.
\end{proposition}
\begin{proof}
We will call such component \emph{basepoint free} and prove the claim by induction on $\rho(X)$. For $\rho(X)=1$ it follows from Corollary \ref{lem-rho-1}. So assume that it is proved for maximal log Fano manifolds with Picard number $\leq m - 1$ and let $\rho(X)=m$. By Theorem \ref{the-theorem}, there exists a free stratum $L$-line on $X$ whose contraction $\pi$ induces the structure of a $\mathbb{P}^k$-bundle on $X$ over a log Fano manifold $(D, \Delta_D)$ of lower dimension. Furthermore, $\rho(D)=m-1$, where $D$ can be identified with some stratum of $\Delta$. Then on $D$ there exists a basepoint free component of $\Delta_D$. Then its pullback to $X$ is basepoint free.
\end{proof}

Now, we are ready to prove Theorem \ref{the-theorem}.\\

\emph{Proof of Theorem \ref{the-theorem}}.
Consider a maximal log Fano manifold $(X, \Delta)$ of dimension $n$. The theorem holds for $n=2$, see Section \ref{sec-surface-pairs}, so we assume that $n\geq 3$. Any component $D_i$ of $\Delta$ is a maximal log Fano manifold of dimension $n-1$ considered as the pair $(D_i, \Delta_{D_i})$. So, by induction the theorem holds for any $D_i$, and more generally for any stratum of $\Delta$. By Proposition \ref{prop-free-curve}, there exists a free stratum $L$-line on $X$, so \ref{the-theorem}(\ref{thm-free-curve}) is proved. Possibly reordering the components of $\Delta$, we have $l = D_2\cap\ldots\cap D_n$. So $D_1$ is a unique component of $\Delta$ such that $l\not\subset D_1$. By Proposition \ref{prop-curve-equiv}, we have that $\mathrm{N}_1(X)$ is generated by $l$ and by curves that belong to $D_1$. By Corollary \ref{cor-picard-bound} either $\rho(X)=\rho(D_1) + 1$ or $\rho(X)=\rho(D_1)$. We consider two cases:\\

\textbf{Case $1$}: $\rho(X)=\rho(D_1)+1$ and $\mathrm{N}_1(X)\simeq \mathrm{N}_1(D_1)\oplus\langle l\rangle$. Then, Proposition \ref{prop-p1-bundle} implies that the curve $l$ is extremal in $\overline{\mathrm{NE}}(X)$, so there exists a $(K_X+\Delta)$-negative contraction $\pi$ of $l$, and this contraction is a $\mathbb{P}^1$-bundle. Hence, Theorem \ref{the-theorem}(\ref{thm-projective-bundle}) is proved in this case. Proposition \ref{prop-stratum-lines} implies that the Mori cone of $X$ is generated by stratum $L$-lines, and hence they are the only extremal curves on $X$, so Theorem \ref{the-theorem}(\ref{thm-stratum-lines}) is proved. Theorem \ref{the-theorem}(\ref{thm-toric-pair}) is proved in Proposition \ref{prop-toric-pair}. It follows that the $1$-complement constructed in Proposition \ref{prop-log-smooth-complement} is toric, so Theorem \ref{the-theorem}(\ref{thm-gamma}) is proved. 

Theorem \ref{the-theorem}(\ref{thm-picard}) is proved as follows. By Theorem \ref{the-theorem} applied to $D_1$, $\Pic (D_1)$ is generated by $D_i|_{D_1}$ for $i\geq 2$. By our assumption, we have $\rho(X)=\rho(D_1)+1$. Since $X$ is a $\mathbb{P}^1$-bundle and $\oo_X(D_1)\simeq \oo_{\mathbb{P}(\mathscr{E})}(1)$, we conclude that $D_i$ generate $\Pic (X)$.

For the proof of Theorem \ref{the-theorem}(\ref{thm-nef-component}) see Proposition \ref{prop-nef-component}. Theorem \ref{the-theorem}(\ref{thm-stratum-set}) is proved in Proposition \ref{prop-stratum-set}. Theorem \ref{the-theorem}(\ref{thm-projective-bundle}) is proved in Proposition \ref{prop-l-extremal}. Finally, \ref{the-theorem}(\ref{thm-split-bundle}) is proved in Proposition \ref{prop-toric-pair}. The proof in the case $\rho(X)=\rho(D_1)+1$ is complete.\\

\textbf{Case $2$}: $\rho(X)=\rho(D_1)$ and $\mathrm{N}_1(X)\simeq \mathrm{N}_1(D_1)$ where this isomorphism is induced by the inclusion $D_1\subset X$. Then Proposition \ref{prop-l-equiv}, implies that the curve $l$ is equivalent to some stratum $L_{D_1}$-line on $D_1$. In Proposition \ref{prop-stratum-lines2}, we proved that the Mori cone of $D_1$ is isomorphic to the Mori cone of $X$, so Theorem \ref{the-theorem}(\ref{thm-stratum-lines}) is proved. It follows that there exists an extremal contraction $\pi$ of $l$ on $X$. In Proposition \ref{prop-projective-bundle2}, we proved that $\pi$ is a $\mathbb{P}^k$-bundle over some stratum $D$, for $k\geq 2$, so Theorem \ref{the-theorem}(\ref{thm-projective-bundle}) is proved in this case. Theorem \ref{the-theorem}(\ref{thm-toric-pair}) and \ref{the-theorem}(\ref{thm-split-bundle}) is proved in Proposition \ref{prop-toric-pair2}. Theorem \ref{the-theorem}(\ref{thm-gamma}) is proved in Proposition \ref{prop-log-smooth-complement2}. Theorem \ref{the-theorem}(\ref{thm-picard}) follows from the corresponding statement for $D$ and the fact that $D_1$ generates $\Pic (X)/\Pic (D)$. For the proof of Theorem \ref{the-theorem}(\ref{thm-nef-component}) see Proposition \ref{prop-nef-component}. Finally, theorem \ref{the-theorem}(\ref{thm-stratum-set}) is proved in Proposition \ref{prop-stratum-set}. The proof of the theorem is complete.

\section{Maximal degeneration}
\begin{definition}{\em  
By a \emph{semistable family} $p\colon \mathscr{X}\to Z\ni o$ over a smooth curve germ $o\in B$, we mean a flat projective morphism from a smooth manifold $\mathscr{X}$, such that the fiber over any point $p$ in $B-\{ o\}$ is smooth, and the fiber $p^{-1}(o)=\sum X_i$, called a \emph{special fiber} of $p$, is an snc variety.}
\end{definition}

Let $X=\sum X_i$ be a special fiber of a semistable family $p\colon\mathscr{X}\to B$. Then the following statement holds.

\begin{lemma}
\label{lem-d-semistability}
For any $i\neq j$ put $D=X_i \cap X_j$. Then 
\begin{equation}
\label{dss}
\mathscr{N}_{D/ X_j} \otimes \mathscr{N}_{D/ X_i} \otimes \bigotimes_{k\neq i, j} \oo_D(X_k|_D) \simeq \oo_D.
\end{equation}
\end{lemma}
\begin{proof}
Since $\sum X_k = p^{-1}(o) \sim 0$ over $B$ one has $\sum X_k|_D = 0$. Note that $\oo_D(X_i|_{D}) \simeq \mathscr{N}_{D/ X_j}$ and $\oo_D(X_j|_{S}) \simeq  \mathscr{N}_{D/ X_i}$, so the claim follows.
\end{proof}

\begin{definition}[{\cite{Fr83}}]
\label{dss-defin}
{\em We refer to the equation \eqref{dss} as the \emph{$d$-semistability condition}.} 
\end{definition}

In fact, \eqref{dss} is just a corollary of the original condition proposed in \cite{Fr83}, but it is enough for our purposes. By a \emph{semistable family of Fano varieties}, we mean a semistable family 
\[
p\colon\mathscr{X}\to B\ni o
\] 
such that $-K_{\mathscr{X}}$ is ample over $B$. Assume that the general fiber of $p$ has dimension $n\geq 1$. Then the special fiber $X=\sum X_i$ is an \emph{snc Fano variety} which means that the line bundle $-K_X$ is ample. From \cite{Lo19}*{3.3} it follows that the dual complex $\mathcal{D}(X)$ is a simplex of dimension $\leq n$. We say that $X$ is a \emph{semistable degeneration} of the general fiber of $p$. We say that $X$ is a \emph{maximal semistable degeneration} if $\dim \mathcal{D}(X) = n$. The condition that $X$ is an snc Fano variety implies that each component $X_i$ is a log Fano with respect to the boundary $\Delta_{X_i}=\sum_{j\neq i} X_j|_{X_i}$. Similarly, each stratum $D$ of $X=\sum X_i$ is a log Fano with respect to the boundary $\Delta_D$ defined analogously. If $X$ is a maximal semistable degeneration then each stratum is a maximal log Fano manifold. We say that a snc Fano variety $X=\sum X_i$ is \emph{maximal} if $\dim \mathcal{D}(X)=n$, where $n$ is the dimension of $X$. We will prove that a maximal semistable degeneration of Fano varieties is unique, up to isomorphism.
That is, $X$ is unique up to isomorphism as an snc variety. From \cite{Tz15} it follows that this is equivalent to the claim that a maximal d-semistable snc Fano variety $X=\sum X_i$ is unique.

\begin{lemma}
\label{lem-unique-free-line}
We work in the notation of Section \ref{section-geometry}. Suppose that on a maximal log Fano manifold $(X, \Delta)$ there exists a unique free stratum $L$-line $l$. Then the corresponding contraction is a $\mathbb{P}^1$-bundle.
\end{lemma}
\begin{proof}
By Theorem \ref{the-theorem}(\ref{thm-projective-bundle}), the contraction $\pi$ of $l$ induces a $\mathbb{P}^k$-bundle structure on $X$. By Theorem \ref{the-theorem}(\ref{thm-stratum-set}), some fiber $F$ of $\pi$ coincides with a stratum of $\Delta$. If $k\geq 2$, then $\dim F \geq 2$ and so there exists some stratum $L$-line $l'\neq l$, so that $l'\subset F$ is contracted by $\pi$. Since $\pi$ is of fiber type, the normal bundle to $l'$ in $X$ is semiample and hence $l'$ is free. This shows that $k=1$, that is, $\pi$ is a $\mathbb{P}^1$-bundle.
\end{proof}

Let $X = \sum X_i$ be a maximal snc Fano variety with $\dim X_i = n$. Then each $(X_i, \Delta_{X_i})$ is a maximal log Fano manifold where $\Delta_{X_i}=\sum_{j\neq i} X_j|_{X_i}$. The divisor $L_{X_i}=-K_{X_i} - \Delta_{X_i}$ is ample. Similarly, each stratum $D$ is a maximal log Fano manifold with respect to the boundary $\Delta_D$ defined by adjunction. Moreover, the divisor $L_D = -K_D - \Delta_D$ is ample. Let us denote the intersection $X_i\cap X_j$ by $D_{i, j} = D_{j, i}$. From now on, we assume that $X=\sum X_i$ is d-semistable.

\begin{proposition}
\label{prop-unique-unique}
Any component $X_i$ has a unique free stratum $L_i$-line and a unique basepoint free component of $\Delta$.
\end{proposition}

\begin{proof}
Without loss of generality let $i=1$, so consider a component $X_1$. By Theorem \ref{the-theorem}(\ref{thm-nef-component}), there exists a basepoint free divisor on $X_1$ which is a component of $\Delta_{X_1}$. Up to reordering, we may assume that this component is $D = D_{1,2}$, that is, $D = X_1\cap X_2$. So $\mathscr{N}_{D/X_1}$ is nef. By d-semistability \eqref{dss}, we obtain
\begin{equation}
\label{eq-dss2}
\mathscr{N}_{D/X_2} \simeq \mathscr{N}_{D/X_1}^{-1} \otimes \bigotimes_{i\geq 3} \oo_D(X_i|_{D})^{-1} \simeq \mathscr{N}_{D/X_1}^{-1} \otimes \oo_D(-\Delta_D)
\end{equation}
where $\Delta_D$ is a boundary on $D$ such that $(D, \Delta_D)$ is a maximal log Fano pair. On $X_2$, we have that $\oo_{X_2}(D|_D) \simeq \mathscr{N}_{D/X_2}$. Note that for any stratum $L_{X_2}$-line $l$ on $X_2$, the intersection $D\cdot l$ is equal to $D|_D\cdot l$ where $l$ is considered as a stratum $L_{D}$-line on $D$.

We prove that there exists a unique free stratum line on $X_2$. Let $l$ be such a line. Assume that $l\subset D$. Then by decomposition~\eqref{eq-normal-bundles} it is free on $D$ and $\deg \mathscr{N}_{D/X_2}|_{l}\geq 0$. Note that this number is equal to $D|_D \cdot l$. By~\eqref{eq-dss2}, we have that $D|_D = - \Delta_D - B$ where $\oo_{D}(B)=\mathscr{N}_{D/X_1}$ is nef on $D$. Thus $(- \Delta_D - B)\cdot l \geq 0$ and hence $(-\Delta_D) \cdot l \geq 0$, so $\Delta_D\cdot l \leq 0$. Since $l$ is free, it follows that $l$ intersects any component of $\Delta_D$ trivially. By Theorem \ref{the-theorem}(\ref{thm-picard}), these components generate $\Pic (D)$, so $l\equiv 0$ on $D$ which is a contradiction. This contradiction shows that no stratum $L_{X_2}$-line that is contained in $D$ is free on $X_2$. By Theorem \ref{the-theorem}(\ref{thm-free-curve}), there exists at least one free stratum $L_{X_2}$-line on $X_2$. We conclude that there exists exactly one such line, and it does not belong to $D$.

Let $\mathcal{G}(X)$ be the graph consisting of $0$-dimensional and $1$-dimensional strata of
the dual complex $\mathcal{D}(X)=\mathcal{D}(\sum X_i)$.
We give $\mathcal{G}(X)$ the structure of a directed graph with single-headed arrows.
We proceed inductively.
We denote by $v_i$ the vertex corresponding to $X_i$.
We pick a vertex $v_i$.
We know that there exists a base point free boundary component $D_{i,j}=X_i\cap X_j$ on $X_i$.
Then, we direct the edge $\overline{v_iv_j}$ from $v_i$ to $v_j$.
Note that on $X_j$ there is a unique free stratum $L_{X_j}$-line which is not contained in $D_{i,j}$.
We proceed inductively from the vertex corresponding to $v_j$.
Observe that $D_{i,j}$ is not nef on $D_j$, hence the graph is not double-headed.
Moreover, two arrows do not point to the same vertex.
Hence, we have an induced decomposition of $\mathcal{G}(X)$ into directed loops.
In particular, on each component $X_i$ there is exactly one basepoint free boundary component. 
\end{proof}

\begin{corollary}
\label{cor-bundle-comp}
On each $X_i$ there exists a unique extremal contraction of fiber type. This contraction is a $\mathbb{P}^1$-bundle. 
\end{corollary}
\begin{proof}
Follows from Proposition \ref{prop-unique-unique} and Lemma \ref{lem-unique-free-line}. 
\end{proof}

\begin{proposition}
For any $j$ one has $\rho(X_j) = n$. Consequently, each $X_j$ is a Bott tower, that is, an iterated $\mathbb{P}^1$-bundle over a point. Moreover, $X_j$ has unique extremal contraction of fiber type which is a $\mathbb{P}^1$-bundle, while the other $n-1$ extremal contractions are either birational or divisorial. 
\end{proposition}
\begin{proof}
By Corollary \ref{cor-bundle-comp}, on each component $X_j$ with $1\leq j\leq n+1$, there exists a unique contraction of fiber type $\pi_j$. Moreover, each $\pi_j$ is a $\mathbb{P}^1$-bundle given by the contraction of a free stratum line $l_j$ on $X_j$.
We claim that for a fixed $j$ there exists a unique $i$ with $i\neq j$, such that $D_{i, j} = X_i\cap X_j$ is a section of $\pi_i$. Indeed, in the proof of Proposition \ref{prop-unique-unique}, a directed edge was drawn from $X_j$ to $X_i$ if and only if $X_i\cap X_j$ is a section of $\pi_i$. We proved that at each $X_j$ exactly one arrow starts, and the claim follows.

For any other $k\neq i,j$ the intersection $D_{k, j}=X_k\cap X_j$ is composed of fibers of $\pi_k$. Hence the unique free stratum line $l_k$ on $X_k$ is contained in $D_{k, j}$ for some $k\neq i, j$. By decomposition \eqref{eq-normal-bundles}, this stratum line is also free on $D_{k, j}$. We claim that $l_k$ is not free on $X_j$. This follows from Proposition \ref{prop-unique-unique}. Indeed, we have  proved that on each component there exists a unique free stratum line and each stratum line is free on some component. Since $l_k$ is free on $X_k$, it is not free on $X_j$. By Proposition~\ref{prop-free-criterion}, we have that $D_{k, j}\cdot l_k < 0$ on $X_j$. 
It follows that the extremal contraction of $l_k$ on $X_j$ is birational. Furthermore, its exceptional locus coincides with $D_{k,j}$. Note that 
\begin{equation}
\label{eq-normal-bundle-on-max}
\mathscr{N}_{l_k/X_j} \simeq  \oo_{\mathbb{P}^1}^{\oplus (n-1)}\oplus \oo_{\mathbb{P}^1}(-a_k)
\end{equation}
for some $a_k\geq1$. We will show that $X_j$ is a Fano manifold and hence by adjunction formula $a_k = 1$.

By Theorem \ref{the-theorem}, each extremal contraction on $X_j$ is a $\mathbb{P}^1$-bundle on some stratum of $\Delta_{X_j}$. From the above discussion it follows that the stratum lines on $X_j$ define $n$ different extremal contractions. By Theorem~\ref{the-theorem}, the Mori cone of $X_j$ is generated by stratum lines. Hence, all the stratum lines correspond to different classes in the Mori cone $\overline{\mathrm{NE}}(X_j)$. We show that $\rho(X_j)=n$ using the fact $n$ stratum lines on $X_j$ correspond to different contractions. We will prove this statement by induction on the number of stages in a generalized Bott tower. 
Let $D$ be the unique boundary component that is a section of $\pi_j$. By Theorem~\ref{the-theorem}, we have $\rho(X_j)=\rho(D) + 1$. The claim is straightforward for one-stage generalized Bott tower which in this case is isomorphic to $\mathbb{P}^1$. Note that on $D$ we have $n-1$ different contractions, since extremal contractions on $D$ are $\mathbb{P}^1$-bundles on different strata of $\Delta_{D}$. Hence, we can apply the inductive hypothesis to conclude that $\rho(D)=n-1$. Hence, $\rho(X_j)=n$. Now, we apply Corollary \ref{cor-bott-tower} to finish the proof.
\end{proof}

\begin{corollary}
\label{cor-trivial-normal-bundle}
The unique basepoint free boundary component on each $X_j$ has trivial normal bundle.
\end{corollary}
\begin{proof}
We have proved that each $X_j$ is a Bott tower, that is an iterated $\mathbb{P}^1$-bundle over a point. Take an iterated pullback of a point. This divisor is a basepoint free boundary component on $X_j$.
Furthermore, it has trivial normal bundle. Since a basepoint free boundary component is unique on $X_j$, the claim follows. 
\end{proof}

\begin{proposition}
\label{prop-each-fano}
Each $X_j$ is a Fano manifold.
\end{proposition}
\begin{proof}
Fix a component $X_j$. By Proposition \ref{prop-unique-unique}, there exists a component $X_i$ such that $D=X_i\cap X_j$ is basepoint free on $X_i$. Moreover, for the unique free stratum line $l$ on $X_j$, we have that $D\cdot l = 1$. All the other stratum lines on $X_j$ are contained in $D$. By Theorem~\ref{the-theorem}, in order to show that $X_j$ is Fano it suffices to prove that $-K_{X_j}$ intersects positively with any stratum line. By isomorphism~\eqref{eq-dss2}, Corollary \ref{cor-trivial-normal-bundle}, and the fact that $D|_D=0$ on $X_i$, we have that $D|_D = -\Delta_D$ on $X_j$. Let $\pi_j$ be the $\mathbb{P}^1$-bundle structure on $X_j$, so $D$ is its section. Note that $\Delta_{X_j} = D + \pi_j^*\Delta_D$. Then, we have that
\[
\Delta_{X_j}|_D = D|_D + \Delta_D = -\Delta_D + \Delta_D = 0.
\]  
For any stratum $L_{X_j}$-line $l$ contained in $D$ the above computation yields 
\begin{equation}
\begin{split}
1 = (-K_{X_j}-\Delta_{X_j})\cdot l &= (-K_{X_j} - \Delta_{X_j})|_D\cdot l \\
= ( -K_{X_j} )|_D \cdot l &= (-K_{X_j})\cdot l.
\end{split}
\end{equation}
Consider the stratum $L_{X_j}$-line $l$ such that $l\not\subset D$, so $D\cdot l = 1$. The contraction of $l$ gives the $\mathbb{P}^1$-bundle structure $\pi_j$ on $X_j$. Therefore $(-K_{X_j})\cdot l = 2>0$. Hence, $X_j$ is a Fano manifold.
\end{proof}

\begin{proposition}
\label{prop-comp-is-unique}
Each $X_j$ is unique up to isomorphism.
\end{proposition}
\begin{proof}
We have established the isomorphism \eqref{eq-normal-bundle-on-max}. 
Using the fact that $X_j$ is Fano, we obtain that $a_k=1$. Exactly one of the stratum lines on $X_j$ has trivial normal bundle. Moreover, this stratum line does not belong to $D$ where $D$ is as in Proposition \ref{prop-each-fano}. The other $n-1$ stratum lines on $X_j$ have the normal bundle of the form $\oo(-1)\oplus \oo^{\oplus(n-2)}$. 

We show that on $D$ there is a unique basepoint free boundary component. Indeed, otherwise taking their pullbacks via the unique fiber type contraction $\pi_j$ on $X_j$, we would have that $X_j$ has more that one basepoint free boundary components. Furthermore, $D$ has a unique free stratum line. Otherwise, we would have two different $\mathbb{P}^1$-bundle structures on $D$ that would yield two basepoint free boundary components on it.

The intersection number of $K_D$ with a stratum line on it is either $-2$ or $-1$.
Moreover, it is $-2$ for a unique stratum line (because $D$ has a unique free stratum line).
In particular, $D$ is a Fano manifold. 
The intersection of stratum lines with $-K_D-\Delta_D$ is $1$, hence the intersection of $\Delta_D$ with a stratum line on $D$ is either $1$ or $0$. Since $\oo_D(D|_D)=\mathscr{N}_{D/X}$, the intersection of $D|_D$ with a stratum line on $D$ is either $-1$ or $0$, see Proposition~\ref{prop-free-criterion}.
Moreover, $D|_D$ intersects trivially only with a unique free stratum line on $D$, so we have $D|_D = -\Delta_D$. 
Therefore by~\cite{Ch19}*{Lemma 3.8}, we have that 
\[
X_j \simeq \mathbb{P}_D(\oo_D\oplus \oo_D(-\Delta_{D})).
\] 
If $n=2$, we get that $X_j\simeq \mathbb{F}_1$, so we may assume that $n\geq 3$. Recall that $D$ is a Bott tower since $\rho(D)=n-1$. In particular, it is a $\mathbb{P}^1$-bundle over its boundary component $D'$ which is defined uniquely since $D$ has a unique basepoint free boundary component. As in the above discussion, we check that $D'$ has a unique free stratum line, a unique basepoint free boundary component, that $D'$ is Fano, and the normal bundles to the stratum lines are either $\oo^{\oplus(n-2)}$ (for a unique stratum line), or $\oo^{\oplus(n-3)}\oplus\oo(-1)$. Finally, $D'|_{D'}=-\Delta_{D'}$. Then $D \simeq \mathbb{P}_{D'}(\oo_{D'}\oplus \oo_{D'}(-\Delta_{D'})).$ We conclude that $X_j$ is an iterated $\mathbb{P}^1$-bundle over a point, where at each stage we take a projectivization of $\oo\oplus \oo(-\Delta)$. Moreover, this procedure is defined uniquely, so $X_j$ is unique.
\end{proof}

\begin{corollary}
\label{glue-unique-way}
A maximal d-semistable snc Fano variety $X=\sum X_j$ is unique up to isomorphism.
\end{corollary}
\begin{proof}
We show that the components $X_i$ can be glued in a unique way into a d-semistable snc Fano variety $X$. It is enough to prove that, using notation as in the proof of Proposition \ref{prop-unique-unique}, $\mathcal{G}(X)$ consists of one cycle. Indeed, if not, then there are at least two cycles in $\mathcal{G}(X)$. Note that any such cycle has length $\geq 3$. Consider one of them. Without loss of generality assume that it consists of the components $X_1, \ldots, X_k$ with $k\geq 3$. We can restrict these components to the intersection $D=X_{k+1}\cap \ldots \cap X_{n+1}$ which has dimension $\geq 3$. Then $X_i|_D=D_i$, for $1\leq i\leq k$, are the components of the boundary divisor on $D$. Moreover, since they form a cycle in $\mathcal{G}(X)$, we know that there exists a unique stratum line on $D_{i+1}$ that is not contained in $D_i\cap D_{i+1}$ is free on $D_{i+1}$ and define divisorial contractions on $D$ (see the proof of Proposition \ref{prop-unique-unique}). More precisely, these contractions are  $\mathbb{P}^1$-bundles with exceptional set $D_i\subset D$. This contradicts the fact that on $D$ there exists a free stratum line (which induces a contractions of fiber type on $D$). Hence in $\mathcal{G}(X)$ there exists only one cycle. It follows that the components $X_i$ can be glued in a unique way.
\end{proof}

\begin{corollary} 
\label{cor-unique-component}
Each $X_j$ is isomorphic to the blow-up of $\mathbb{P}^n$ in a flag of subspaces 
\[
p \in \mathbb{P}^1\subset \ldots \subset \mathbb{P}^{n-2}\subset \mathbb{P}^n.
\]
\end{corollary}
\begin{proof}
By Corollary \ref{cor-unique-component}, each $X_j$ is isomorphic to a unique Fano manifold which is determined by the dimension. Denote the blow-up of $\mathbb{P}^n$ in the given flag of subspaces by $X'$. We obtain it is as follows: first, we blow-up a point on $\mathbb{P}^n$, then blow-up a strict preimage of a line passing trough this point, and so on. By induction using Proposition \ref{prop-comp-is-unique}, it is enough to show that $X'$ is isomorphic to $\mathbb{P}_D(\oo_D \oplus \oo_D(-\Delta_D))$ where $D$ is a blow-up of flag of subspaces on $\mathbb{P}^{n-1}$, and $\Delta_D$ is a boundary such that $(D, \Delta_D)$ is a maximal log Fano manifold. This can checked directly.
\end{proof}

\emph{Proof of Theorem \ref{thm-2}}.
Let $X=\sum X_i$ be a maximal d-semistable snc Fano variety. By Corollary, \ref{cor-unique-component} each of its components is isomorphic to the blow-up of $\mathbb{P}^n$ in a flag of subspaces. By Corollary \ref{glue-unique-way}, they can be glued with each other in a unique way. On the other hand, such $X$ is the special fiber of the following semistable family of Fano varieties. Start with a trivial family $\mathbb{P}^n\times B\to B\ni o$ and blow-up a flag of subspaces $\mathbb{P}^0\subset \mathbb{P}^1\subset \cdots \subset \mathbb{P}^{n-2}$ in a fiber over $o\in B$. See \cite{Hu06}*{4.3} for details of this construction. The fact that semistable degeneration of Fano manifolds is unique follows from \cite{Tz15}. The theorem is proved.

\begin{appendix}
\section{Generalized Bott towers}
\label{section-bott-towers}

\subsection{Toric varieties}
We follow the notation in \cite{CLS11}. Let $T$ be an algebraic torus of dimension $n$, that is $T = (\mathbb{K}^\times)^n$. By a \emph{toric variety}, we mean a normal variety $X$ such that $T$ acts on $X$ with a Zariski open orbit which is isomorphic to $T$. Let $M$ be a character lattice of $T$, so $M\simeq \mathbb{Z}^n$ and $T = \mathrm{Spec}\,\mathbb{K}[M]$, and let $N=\mathrm{Hom}(M, \mathbb{Z})$ be a dual lattice, so $N\simeq \mathbb{Z}^n$ and there is a non-degenerate pairing $M \times N \to \mathbb{Z}$. Put $M_{\mathbb{Q}}=M\otimes \mathbb{Q}$ and $N_{\mathbb{Q}}=N\otimes \mathbb{Q}$. Then the pairing naturally extends to a non-degenerate pairing $M_{\mathbb{Q}}\otimes N_{\mathbb{Q}}\to \mathbb{Q}$. 

Let $\sigma \subset N_{\mathbb{Q}}$ be a pointed convex polyhedral cone. Then its dual $\sigma^{\vee} \subset M_{\mathbb{Q}}$ defines a monoid algebra $\mathbb{K}[\sigma^{\vee} \cap M]$ and an affine toric variety $X_\sigma = \mathrm{Spec}\,\mathbb{K}[\sigma^{\vee} \cap M]$. By a \emph{fan} $\Sigma$ we mean a finite set of pointed convex polyhedral cones in $N_{\mathbb{Q}}$ such that the face of any cone in $\Sigma$ is again in $\Sigma$, and the intersection of two cones in $\Sigma$ is a face of both. To a fan $\Sigma$ we can associate a toric variety $X_\Sigma$ by gluing the affine toric varieties $X_\sigma$ and $X_{\sigma'}$ along $X_{\sigma \cap \sigma'}$ whenever $\sigma$ and $\sigma'$ are cones of $\Sigma$. Any $k$-dimensional face of $\sigma$ defines a $(n - k)$-dimensional orbit of $X (\sigma)$. The above gluing is along $T$-invariant subvarieties, so $X_\Sigma$ is also endowed with a $T$-action that contains an open orbit isomorphic to $T$, so $X_\Sigma$ is toric. We denote by $\Sigma(k)$ the set of $k$-dimensional faces of $\Sigma$.


\subsection{Generalized Bott towers}
\begin{definition}{\em 
We follow~\cite{Ch19} and~\cite{Su20}. A \emph{generalized Bott tower} $X = \{ X_i\, |\, i = 0, \ldots , m\}$ with $m$ stages (or of height $m$) is a sequence 
\[
X_m \xrightarrow{\pi_m} X_{m-1} \xrightarrow{\pi_{m-1}} \ldots \xrightarrow{\pi_2} X_1 \xrightarrow{\pi_1} X_0 = \text{pt}
\]
where for any $i$, $1\leq i\leq m$, one has $X_i \simeq \mathbb{P}_{X_{i-1}}(\mathscr{E})$ and $\mathscr{E} = \oo_{X_{i-1}} \oplus \mathscr{L}_i^{(1)}\oplus \ldots \oplus \mathscr{L}_i^{(n_i)}$ for some line bundles $\mathscr{L}_i^{(1)}, \ldots, \mathscr{L}_i^{(n_i)}$ on $X_{i-1}$ and some $n_i\geq 1$. We call $X_i$ \emph{stage $i$} of the generalized Bott tower $X$. When $n_i = 1$ for any $i$, $1\leq i \leq m$, the generalized Bott tower is called the \emph{Bott tower}. We denote by $n=n_1+\ldots+n_m$ the dimension of $X$.}
\end{definition}

It is well known that $\Pic (X_i) \simeq \mathbb{Z}^i$ for any $i$, $1\leq i \leq m$. Hence each line bundle $\mathscr{L}_{i}^{(k)}$ where $1 \leq k \leq n_i$ over $X_{i - 1}$ is defined by $i - 1$ integers $a_{i, 1}^{(k)}, \ldots , a_{i, i-1}^{(k)}$. Thus, a generalized Bott tower is completely described by the integers
\begin{align}
\nonumber & m  \geq 1  \\
\nonumber & n_i  \geq 1  \text{ for }  1\leq i\leq m\\
\nonumber & a_{i, j}^{(k)}   \text{ for }  1\leq i\leq m, 1\leq j\leq i-1,
\text{ and } 1\leq k\leq n_i.
\end{align}

\begin{example}{\em 
We give some examples of Bott towers:
\begin{enumerate}
\item
A projective space $\mathbb{P}^n$ is an $1$-stage generalized Bott tower, so $m=1$ and $n_1 = n$.
\item
A product of projective space $\mathbb{P}^{n_1}\times \ldots \times \mathbb{P}^{n_m}$ for $m\geq 1$ is an $m$-stage generalized Bott tower. Here $n_1, \ldots, n_m$ are some positive numbers and $a_{i, i-1}^{k}=0$ for any $k$, $1\leq k\leq n_i$ and any $i$, $1\leq i\leq m$.
\end{enumerate}
}
\end{example}

From Proposition $7.3.3$ in \cite{CLS11}, it follows that each stage $X_i$ of a $m$-stage generalized Bott tower is a toric variety and each $\pi_i$ is a toric morphism for $1\leq i\leq m$. Let us describe the fan $\Sigma$ of $X$. It will have $n+m$ rays. Let $\{ e_1^1, \ldots, e_1^{n_1}, \ldots , e_m^1, \ldots, e_m^{n_m} \}$ be a standard basis of $\mathbb{Q}^n$. Define the rays $u_1^0, u_1^1,\ldots, u_1^{n-1}, \ldots u_m^{0}, \ldots, u_m^{n_m}$ in $\mathbb{Z}^n$ as follows:
\[
u_i^{k} = e_i^k\ \text{for}\ 1\leq k \leq n_i, \quad \quad u_i^0 = - \sum_{k=1}^{n_i} e_i^k + \sum_{j=i+1}^{m} \sum_{k=1}^{n_j} a_{j, i}^{(k)}e_j^k.
\]
Define $n$-dimensional cones of $\Sigma$ as follows:
\[
\sum_{i=1}^{m} ( \mathbb{Q}_{\geq 0} u_i^0 + \ldots + \widehat{\mathbb{Q}_{\geq 0} u_i^{k_i}} + \ldots + \mathbb{Q}_{\geq 0} u_i^{n_i} )\subset \mathbb{Q}^n.
\]
for each $1\leq k_i\leq n_i$.

\begin{example}\label{ex:hirz}{\em 
A Hirzebruch surface $\mathbb{F}_l$ for $l\geq 0$ is a $2$-stage Bott tower with $m=2$, $n_1 = n_2 = 1$, and $a_{2,1}^{(1)} = l$. Thus, we have a fan in $\mathbb{Q}^2=\langle e_1^1, e_2^1 \rangle$ with the rays 
\[
u_1^1= e_1^1, \quad u_2^1 = e_2^1, \quad u_1^0 = -e_1^1 + a_{2, 1}^{(1)} e_2^1 , \quad u_2^0 = -e_2^1.
\]
We have four $2$-dimensional cones, namely
\[
\mathbb{Q}_{\geq 0} u_1^1 + \mathbb{Q}_{\geq 0} u_2^1, \quad \mathbb{Q}_{\geq 0} u_1^0 + \mathbb{Q}_{\geq 0} u_2^1, \quad \mathbb{Q}_{\geq 0} u_1^1 + \mathbb{Q}_{\geq 0} u_2^0, \quad \mathbb{Q}_{\geq 0} u_1^0 + \mathbb{Q}_{\geq 0} u_2^0.
\]

Note that $(\mathbb{F}_l,s_l)$ is log smooth Fano, 
where $s_l$ is the negative section of the Hirzebruch surface.
This example shows that log Fano manifolds in dimension two are not bounded.
}
\end{example}
\end{appendix}

\end{document}